\documentclass[11pt]{amsart}
\usepackage{amsmath} 
\usepackage{amssymb} 
\usepackage{amsthm}
\usepackage{amsfonts}            
\usepackage{mathrsfs}
\usepackage[all]{xy}
\usepackage{mathtools}
\usepackage{mathabx}
\usepackage{colonequals}

\usepackage{amsmath} 
\usepackage{amssymb} 
\usepackage{amsthm}
\usepackage{amsfonts}            
\usepackage{mathrsfs}
\usepackage{amsthm,amsfonts,amssymb,amsmath,amsxtra,mathtools}
\usepackage[all]{xy}
\usepackage{mathtools}
\usepackage{tikz-cd}
\usepackage{mathabx}
\usepackage{colonequals}
\usepackage{xcolor}
\usepackage[pagebackref]{hyperref} %% \hypersetup
\hypersetup{  %% Set up hyperref color
	colorlinks=true,
	linkcolor=black,
	citecolor=black,
	urlcolor=magenta,
}

\usepackage{geometry}
\usepackage{setspace}

\title{The Regularity of difference divisors}
\author{Baiqing Zhu}
\date{May 2024}
\begin{document}
\theoremstyle{definition}        
\newtheorem{definition}{Definition}[subsection] 
\newtheorem{example}[definition]{Example}
\newtheorem{remark}[definition]{Remark}

\theoremstyle{plain}
\newtheorem{theorem}[definition]{Theorem}
\newtheorem{lemma}[definition]{Lemma}         
\newtheorem{proposition}[definition]{Proposition}
\newtheorem{corollary}[definition]{Corollary}

\newcommand{\ofb}{\mathcal{O}_{\Breve{F}}}
\newcommand{\zpb}{\Breve{\mathbb{Z}}_{p}}
\newcommand{\zp}{\mathbb{Z}_{p}}
\newcommand{\qpb}{\Breve{\mathbb{Q}}_{p}}
\newcommand{\fb}{\Breve{F}}
\newcommand{\vcrys}{\mathbf{V}_{\textup{crys}}}
\newcommand{\vcrysz}{\mathbf{V}_{\textup{crys},z}}
\newcommand{\bpi}{\boldsymbol{\pi}}
\newcommand{\testleftlong}{\longleftarrow\!\shortmid}

\begin{abstract}
    For a prime number $p>2$ and a finite extension $F/\mathbb{Q}_p$, we explain the construction of the difference divisors on the unitary Rapoport-Zink spaces of hyperspecial level over $\ofb$, and the GSpin Rapoport-Zink spaces of hyperspecial level over $\zpb$ associated to a minuscule cocharacter $\mu$ and a basic element $b$. We prove the regularity of the difference divisors, find the formally smooth locus of both the special cycles and the difference divisors, by a purely deformation-theoretic approach.
\end{abstract}
\maketitle

\pagenumbering{roman}
\tableofcontents
\newpage
\pagenumbering{arabic}

\section{Introduction}
\subsection{Background}
Let $n\geq2$ be a positive integer. Let $\mathcal{N}_{1,n-1}$ be the unitary Rapoport-Zink space of signature $(1,n-1)$ with hyperspecial level. Kudla-Rapoport special divisors $\mathcal{Z}(x)$ on $\mathcal{N}_{1,n-1}$ have been defined and extensively studied in the works of Kudla and Rapoport \cite{KR11}. Then Terstiege introduced the difference divisors on the formal scheme $\mathcal{N}_{1,n-1}$ in \cite{Ter13a}, he also proved the regularity of difference divisors in \cite{Ter13b}. A key step in the proof is a previous result in his joint work with Rapoport and Zhang \cite[Theorem 10.7]{RTZ13}, whose proof is based on the windows theory developed by Zink in \cite{Zin01} and \cite{Zin02}.
\par
For a self-dual quadratic lattice $V$ of rank $m=n+1$ over $\zp$, a Hodge type Rapoport-Zink space $\mathcal{N}(V)$ associated to a minuscule cocharacter $\mu$ and a basic element $b$ has been constructed and studied in the works of Kim \cite{Kim18} and Howard-Pappas \cite{HP17}. When $V$ is one of the self-dual lattices of rank 4, Terstiege defined difference divisors on the formal scheme $\mathcal{N}(V)$ and studied the intersection numbers of them in \cite{Ter11}, where he also proved the regularity of these difference divisors. The notion of difference divisors can be defined similarly on $\mathcal{N}(V)$ for general $V$, but the regularity of these difference divisors was previously not known to us. It's difficult to prove a similar result as \cite[Theorem 10.7]{RTZ13} since the $p$-divisible groups parameterized by the formal scheme $\mathcal{N}(V)$ have dimension $2^{n}$ and height $2^{n+1}$, so the windows theory becomes much more complicated in this case.

\subsection{Main results}
Let $p$ be an odd prime. Let $\mathbb{F}$ be the algebraic closure of the finite field $\mathbb{F}_p$. Let $n\geq2$ be an integer. Let $F/\mathbb{Q}_p$ be a finite extension with uniformizer $\varpi$ in its ring of integers $\mathcal{O}_{F}$. Let $\Breve{F}$ be the completion of the maximal unramified extension of $F$.
\par
In the unitary case, let $E/F$ be an unramified extension. The unitary Rapoport-Zink space $\mathcal{N}_{1,n-1}$ is a formally smooth formal scheme over $\textup{Spf}\,\ofb$ parameterizing hermitian formal $\mathcal{O}_{E}$-modules of signature $(1,n-1)$ within the supersingular quasi-isogeny class ($\S$\ref{rz-unitary}). Denote by $\mathcal{N}^{u}$ the formal scheme $\mathcal{N}_{1,n-1}$ for simplicity. Let $\overline{\mathbb{Y}}$ be the unique (up to isomorphism) hermitian $\mathcal{O}_{E}$-module of signature $(0,1)$ over $\mathbb{F}$. Let $\mathbb{X}$ be the framing hermitian $\mathcal{O}_{E}$-module of signature $(1,n-1)$ over $\mathbb{F}$. The space of space quasi-homomorphisms $\mathbb{V}^{u}\coloneqq\textup{Hom}^{\circ}_{\mathcal{O}_{E}}(\overline{\mathbb{Y}},\mathbb{X})$ carries a natural $E/F$-hermitian form $(\cdot,\cdot)$ ($\S$\ref{special-unitary}).
\par
In the GSpin case, let $V$ be a self-dual quadratic lattice of rank $m=n+1$ over $\zp$. Associated to $V$ we have a local unramified Shimura-Hodge datum $(G,b,\mu,C(V))$, where $G = \textup{GSpin}(V)$, $b\in G(\qpb)$ is a basic element, $\mu: \mathbb{G}_{m}\rightarrow G$ is a certain cocharacter, and $C(V)$ is the Clifford algebra of $V$. This local unramified Shimura-Hodge data gives rise to a supersingular $p$-divisible group $\mathbb{X}_{V}$, and a GSpin Rapoport-Zink space $\textup{RZ}(G, b, \mu, C(V))$ of Hodge type, parametrizing $p$-divisible groups quasi-isogenous to $\mathbb{X}_V$ with crystalline Tate tensors . It is a formally smooth formal scheme over $\zpb$ ($\S$\ref{gspin-rz}). Let $\mathcal{N}(V)$ be a connected component of the formal scheme $\textup{RZ}(G, b, \mu, C(V))$. The space of special quasi-homomorphisms $\mathbb{V}^{o}\coloneqq(V\otimes_{\zp}\qpb)^{\overline{b}\circ\sigma}$ can be identified as a subspace of $\textup{End}^{\circ}(\mathbb{X}_{V})$ ($\S$\ref{special-gspin}), where $\overline{b}$ is the image of $b$ under the homomorphism $G(\qpb)\rightarrow\textup{SO}(V)(\qpb)$, $\sigma\in\textup{Gal}(\Breve{\mathbb{Q}}_p/\mathbb{Q}_p)$ is the arithmetic Frobenius automorphism on $\Breve{\mathbb{Q}}_p$. The $\mathbb{Q}_p$-vector space $\mathbb{V}^{o}$ carries a natural quadratic form $q_{\mathbb{V}^{o}}$ induced from $V\otimes_{\zp}\qpb$. Denote by $\mathcal{N}^{o}$ the formal scheme $\mathcal{N}(V)$.
\par
For any subset $L\subset\mathbb{V}^{u}$ (resp. $L\subset\mathbb{V}^{o}$). The special cycle $\mathcal{Z}(L)$ is a closed formal subscheme of $\mathcal{N}^{u}$ (resp. $\mathcal{N}^{o}$), over which each quasi-homomorphism $x \in L$ deforms to homomorphisms. When $L=\{x\}$ for a nonzero element $x\in\mathbb{V}^{u}$ (resp. $x\in\mathbb{V}^{o}$). The special cycle $\mathcal{Z}(x)\coloneqq\mathcal{Z}(\{x\})$ is a Cartier divisor on the formal scheme $\mathcal{N}^{u}$ (resp. $\mathcal{N}^{o}$) ($\S$\ref{special-cycles}). There is a closed immersion of divisors $\mathcal{Z}(\varpi^{-1}x)\rightarrow\mathcal{Z}(x)$ (resp. $\mathcal{Z}(p^{-1}x)\rightarrow\mathcal{Z}(x)$), define the difference divisor associated to $x$ to be the following Cartier divisor on $\mathcal{N}^{u}$ (resp. $\mathcal{N}^{o}$),
\begin{equation*}
        \mathcal{D}(x)\coloneqq\mathcal{Z}(x)-\mathcal{Z}(\varpi^{-1}x)\,\,\,\,\textup{(resp. $\mathcal{D}(x)\coloneqq\mathcal{Z}(x)-\mathcal{Z}(p^{-1}x)$)}.
\end{equation*}
This definition recovers Terstiege's definition of difference divisors in \cite{Ter13b} (unitary case, $F=\mathbb{Q}_p$), and in \cite{Ter11} (GSpin case, $V$ is a specific rank $4$ self-dual lattice over $\mathbb{Z}_p$).
The main results of this article are stated in the following theorem (Theorem \ref{main-results} and Corollary \ref{finalproof}).
\begin{theorem}
Let $x\in\mathbb{V}^{u}$ (resp. $x\in\mathbb{V}^{o}$) be a nonzero special quasi-homomorphism such that $(x,x)\in\mathcal{O}_{E}$ (resp. $q_{\mathbb{V}^{o}}(x)\in \zp$).
\begin{itemize}
    \item[(i)] There is an equality $\mathcal{D}(x)(\mathbb{F})=\mathcal{Z}(x)(\mathbb{F})$.
    \item[(ii)] The difference divisor $\mathcal{D}(x)$ is a regular formal scheme. It is formally smooth over $\ofb$ (resp. $\zpb$) of relative dimension $n-2$ at a point $z\in\mathcal{D}(x)(\mathbb{F})$ if and only if $z\notin\mathcal{Z}(\varpi^{-1}x)(\mathbb{F})$ (resp. $z\notin\mathcal{Z}(p^{-1}x)(\mathbb{F})$) and there exists a lift of $z$ to $z^{\prime}\in\mathcal{Z}(x)(\ofb/\varpi^{2})$ (resp. $z^{\prime}\in\mathcal{Z}(x)(\zpb/p^{2})$).
    \item[(iii)] Let $z\in\mathcal{Z}(x)(\mathbb{F})$ be a point. Let $\mathcal{D}(x)_{\mathbb{F}}=\mathcal{D}(x)\times_{\ofb}\mathbb{F}$ and $\mathcal{Z}(x)_{\mathbb{F}}=\mathcal{Z}(x)\times_{\ofb}\mathbb{F}$ (resp. $\mathcal{D}(x)_{\mathbb{F}}=\mathcal{D}(x)\times_{\zpb}\mathbb{F}$ and $\mathcal{Z}(x)_{\mathbb{F}}=\mathcal{Z}(x)\times_{\zpb}\mathbb{F}$). Then 
    \begin{equation*}
        \textup{Tgt}_z(\mathcal{D}(x)_{\mathbb{F}})=\textup{Tgt}_z(\mathcal{Z}(x)_{\mathbb{F}}),
    \end{equation*}
    where $\textup{Tgt}_z(\mathcal{D}(x)_{\mathbb{F}})$ and $\textup{Tgt}_z(\mathcal{Z}(x)_{\mathbb{F}})$ are the tangent spaces of the formal schemes $\mathcal{D}(x)_{\mathbb{F}}$ and $\mathcal{Z}(x)_{\mathbb{F}}$ at the point $z$, respectively.
\end{itemize}
\label{main-intro}
\end{theorem}
We give some examples of the equality $\mathcal{D}(x)(\mathbb{F})=\mathcal{Z}(x)(\mathbb{F})$ in Theorem \ref{main-intro} (i). See also Example \ref{d=z}.
\begin{example}
    \begin{itemize}
        \item [(a)] The orthogonal case: Let $H$ be the unique rank $4$ self-dual quadratic lattice over $\zp$ whose discriminant is not a square in $\zp^{\times}$. Let $x\in\mathbb{V}^{o}$ be a non-isotropic vector and $n=\nu_p(q_{\mathbb{V}^{o}}(x))$. Let $x^{\prime}=p^{-[n/2]}\cdot x$. Let $\mathscr{B}$ be the Bruhat-Tits building of $\textup{PGL}_2(\mathbb{Q}_{p^{2}})$. The element $x^{\prime}$ induces an automorphism of $\mathscr{B}$. Let $\mathscr{B}^{x^{\prime}}$ be the fixed set. For a vertex $[\Lambda]$ in the Bruhat-Tits building $\mathscr{B}$ of $\textup{PGL}_2(\mathbb{Q}_{p^{2}})$, let $d_{[\Lambda]}$ be the distance of $[\Lambda]$ to $\mathscr{B}^{x^{\prime}}$.
        \par
        Let $\mathcal{Z}(x)_{\mathbb{F}}=\mathcal{Z}(x)\times_{\zpb}\mathbb{F}$ and $\mathcal{D}(x)_{\mathbb{F}}=\mathcal{D}(x)\times_{\zpb}\mathbb{F}$. By \cite[Theorem 3.6]{Ter11}, we have the following equality of effective Cartier divisors on $\mathcal{N}^{o}_{\mathbb{F}}\coloneqq\mathcal{N}^{o}\times_{W}\mathbb{F}$,
        \begin{equation*}
        \mathcal{Z}(x)_{\mathbb{F}}=\sum\limits_{[\Lambda]:d_{[\Lambda]}\leq(n-1)/2}(1+p+\cdots+p^{(n-1)/2-d_{[\Lambda]}})\cdot\mathbb{P}_{[\Lambda]}+p^{n/2}\cdot s,
        \end{equation*}
        where $\mathbb{P}_{[\Lambda]}$ is a divisor associated to the vertex $[\Lambda]$ and isomorphic to the projective line $\mathbb{P}_{\mathbb{F}}^{1}$. The symbol $s$ represents another divisor which only depends on the element $x^{\prime}$ and is nonempty only if $n$ is even. Since $\mathcal{D}(x)=\mathcal{Z}(x)-\mathcal{Z}(p^{-1}x)$, we have
        \begin{equation*}
            \mathcal{D}(x)_{\mathbb{F}}=\begin{cases}
                \sum\limits_{[\Lambda]:d_{[\Lambda]}=(n-1)/2}\mathbb{P}_{[\Lambda]}+\sum\limits_{[\Lambda]:d_{[\Lambda]}\leq(n-3)/2}p^{(n-1)/2-d_{[\Lambda]}}\cdot\mathbb{P}_{[\Lambda]}+(p^{n/2}-p^{(n-2)/2})\cdot s, &\textup{if $n\geq1$;}\\
                s,&\textup{if $n=0$.}
            \end{cases}
        \end{equation*}
        Therefore $\mathcal{D}(x)(\mathbb{F})=\mathcal{Z}(x)(\mathbb{F})$.
        \item[(b)] The unitary case: Let $n=3$ and $F=\mathbb{Q}_p$. For an element $x\in\mathbb{V}^{u}$, the proof of (\cite[Proposition 2.13]{Ter13a}) implies that if $z\in\mathcal{Z}(x)(\mathbb{F})$, then the local equation of the difference divisor $\mathcal{D}(x)$ is not a unit at the point $z$, hence $\mathcal{Z}(x)(\mathbb{F})\subset\mathcal{D}(x)(\mathbb{F})$. Therefore $\mathcal{D}(x)(\mathbb{F})=\mathcal{Z}(x)(\mathbb{F})$.
    \end{itemize}
\end{example}

\subsection{Strategy and novelty of the proof}
We prove Theorem \ref{main-intro} by a purely deformation-theoretic approach. 
\par
In the unitary case, we have the Grothendieck-Messing deformation theory of hermitian $\mathcal{O}_{E}$-modules of signature $(1,n-1)$. In the GSpin case, similar deformation results (Lemma \ref{deformo}) are also available in Madapusi's work \cite[Proposition 5.16]{Mad16} (see also the works of Li and Zhang \cite[Lemma 4.6.2]{LZ22b}). 
\par
We prove a commutative algebra result Lemma \ref{technical} which helps us to identify a sequence of regular elements in a formally finitely generated and formally smooth algebra over a discrete valuation ring.
\par
Our final proof combines the deformation theory and the commutative algebra results. In stead of analyzing one single difference divisor, we study a sequence of difference divisors $\{\mathcal{D}(\varpi^{a}x)\}_{a\geq0}$ (resp. $\{\mathcal{D}(p^{a}x)\}_{a\geq0}$). We show that the local equations of difference divisors in this sequence satisfy the assumptions in Lemma \ref{technical}. The verification is a series of deformation-theoretic computations based on Lemma \ref{deformu} and Lemma \ref{deformo}, the computations are much easier compared to windows theory and works for general finite extension $F/\mathbb{Q}_p$ since both two lemmas are of Grothendieck-Messing type.

\subsection{Applications}
\subsubsection{Linear invariance}
Li and Zhang \cite[Lemma 2.8.1]{LZ22a} proved the linear invariance of the derived intersections on the Rapoport-Zink space $\mathcal{N}^{u}$ based on the regularity of difference divisors, which guarantees that the local arithmetic intersection numbers on $\mathcal{N}^{u}$ are well-defined. 
\par
The parallel result in the GSpin case is proved by Terstiege \cite[Lemma 4.1, Proposition 4.2]{Ter11} (see also \cite[Lemma 4.11.1]{LZ22b}) using globalization. Our result can be applied to give an alternative proof of the linear invariance of the derived intersections on the Rapoport-Zink space $\mathcal{N}^{o}$ along the lines of \cite[Lemma 2.8.1]{LZ22a} (see \cite[foodnote 3]{LZ22b}),  which guarantees that the local arithmetic intersection numbers on $\mathcal{N}^{o}$ are well-defined.
\subsubsection{Rapoport-Zink spaces with level structures}
In some cases, difference divisors can be identified with Rapoport-Zink spaces with certain level structures. We give several examples.
\begin{itemize}
    \item[(a)] When the valuation of the special quasi-homomorphism $x\in\mathbb{V}^{u}$ (resp. $x\in\mathbb{V}^{o}$) has valuation $0$, i.e., $\nu_{\varpi}((x,x))=0$ (resp. $\nu_p(q_{\mathbb{V}^{o}}(x))=0$). Then the difference divisor $\mathcal{D}(x)=\mathcal{Z}(x)$ is isomorphic to the unitary (resp. GSpin) Rapoport-Zink space of hyperspecial level of one dimension lower (\cite[Remark 4.5]{RSZ} and \cite[Lemma 3.2.2]{LZ17}).\\
        In the GSpin case, when the valuation of the special quasi-homomorphism $x\in\mathbb{V}^{o}$ has valuation $1$, i.e., $\nu_p(q_{\mathbb{V}^{o}}(x))=1$. Then the difference divisor $\mathcal{D}(x)=\mathcal{Z}(x)$ is isomorphic to the GSpin Rapoport-Zink space of almost self-dual level of one dimension lower (\cite[Theorem 4.5]{Oki20}).
    \item[(b)] In the unitary case, let $n=2$. Then $\mathcal{N}^{u}\simeq\textup{Spf}\,\ofb[[t]]$. Let $x\in\mathbb{V}^{u}$ be a special quasi-homomorphism such that $\nu_{\varpi}((x,x))=k\geq0$. Then the difference divisor $\mathcal{D}(x)$ is isomorphic to the quasi-canonical lifting of level $k$ with respect to the unramified extension $E/F$ (\cite[Proposition 8.1]{KR11} for the case $F=\mathbb{Q}_p$ which also works for general field $F$).
    \\
    In the GSpin case, we take $V=\textup{M}_2(\zp)^{\textup{tr}=0}$ equipped with the quadratic form given by the determinant. Then $\mathcal{N}^{o}\simeq\textup{Spf}\,\zpb[[t]]$. Let $x\in\mathbb{V}^{o}$ be a special quasi-homomorphism such that $\nu_{p}(q_{\mathbb{V}^{o}}(x))=k\geq0$. Then the difference divisor $\mathcal{D}(x)$ is isomorphic to the quasi-canonical lifting of level $[k/2]$ with respect to an unramified extension (when $k$ is even) or ramified extension (when $k$ is odd) (\cite[(5.10)]{GK93}, see also \cite[$\S$5.1]{LZ22b}).
    \item[(c)] Let $V=\textup{M}_2(\mathbb{Z}_p)$ with the quadratic form given by the determinant. It is a self-dual $\mathbb{Z}_p$-lattice of rank $4$. In this case, we have $\mathcal{N}^{o}\simeq\textup{Spf}\,\zpb[[t_1,t_2]]$, $\mathbb{V}^{o}$ is the unique division quaternion algebra over $\mathbb{Q}_p$. Let $x\in\mathbb{V}^{o}$ be an element such that $\nu_p(q_{\mathbb{V}^{o}}(x))=k$ for some integer $k\geq0$. The difference divisor $\mathcal{D}(x)$ on $\mathcal{N}^{o}$ is isomorphic to the Rapoport-Zink space associated with the modular curve $\mathcal{X}_0(p^{k})$ (cf. \cite[Theorem 6.2.3]{Zhu23}). A key ingredient in the proof is the regularity of the difference divisor $\mathcal{D}(x)$. Notice that the definition of difference divisor is relatively simple, while the Drinfeld level structure defining the modular curve $\mathcal{X}_0(p^{k})$ (cf. \cite{KM85}, see also \cite{Dr74}) is much more subtle.
\end{itemize}

\subsection{The structure of the paper}
In $\S$\ref{review}, we review the construction and basic properties of the unitary and GSpin Rapoport-Zink spaces. We also give the definition of the difference divisor. In $\S$\ref{def}, we study the deformation theory of the the unitary and GSpin Rapoport-Zink spaces and special cycles on them at a given $\mathbb{F}$-point. The main results are Lemma \ref{deformu} and Lemma \ref{deformo}. In $\S$\ref{smooth}, we find the formally smoothness locus of a special cycle in the deformation-theoretic language compatible with $\S$\ref{def}. In $\S$\ref{regularity-part}, we first prove a commutative algebra result (Lemma \ref{technical}), then we apply it to prove Theorem \ref{main-intro}.

\subsection{Acknowledgments}
The author is grateful to Professor Michael Rapoport and Professor Chao Li for their careful reading of the original manuscript and many helpful comments. The author is also grateful to referees for careful reading and very helpful suggestions. The author is supported by the Department of Mathematics at Columbia University in the city of New York.

\subsection{Notations}
Throughout this article, we fix an odd prime $p$. Let $F/\mathbb{Q}_{p}$ be a finite extension, with ring of integers $\mathcal{O}_{F}$, and uniformizer $\varpi$. Let $\mathbb{F}$ be the algebraic closure of $\mathbb{F}_p$. Let $\fb$ be the completion of the maximal unramified extension of $F$. Let $\sigma\in\textup{Gal}(\Breve{F}/F)$ be the arithmetic Frobenius automorphism of the field $\fb$. Let $\ofb$ be the integer ring of $\fb$. 
\par
For a scheme or formal scheme $S$ over $\ofb$, denote by $\textup{NCRIS}_{\ofb}(S/\textup{Spec}\,\ofb)$ the big fppf nilpotent crystalline site of $S$ over $\ofb$ (cf. \cite[Definition B.5.7.]{FGL08}), the definition is the same as the crystalline site defined in the works of Berthelot, Breen and Messing \cite[$\S$1.1.1.]{BBM82} except that we replace the pd-structure by $\mathcal{O}_F$-pd-structure \cite[Definition B.5.1.]{FGL08} (Notice that they are equivalent when $F=\mathbb{Q}_p$). Denote by $\mathcal{O}_{S}^{\textup{crys}}$ the structure sheaf in this site. For a point $z\in S(\mathbb{F})$, let $\widehat{\mathcal{O}}_{S,z}$ be the complete local ring of $S$ at the point $z$. Let $\textup{Nilp}_{\ofb}$ be the category of $\ofb$-schemes on which $\varpi$ is locally nilpotent. For an object $S$ in $\textup{Nilp}_{\ofb}$, we use $\overline{S}$ to denote the scheme $S\times_{\ofb}\mathbb{F}$. 
\par
Denote by $\textup{Alg}_{\ofb}$ the category of noetherian $\varpi$-adically complete $\ofb$-algebras. Denote by $\textup{Art}_{\ofb}$ the category of local artinian $\ofb$-algebras with residue field $\mathbb{F}$. Denote by $\textup{ANilp}_{\ofb}^{\textup{fsm}}$ be the category of noetherian adic $\ofb$-algebras, in which $\varpi$ is nilpotent, which are formally finitely generated and formally smooth over $\ofb/\varpi^{k}$ for some $k\geq1$.

\section{Special cycles on Rapoport-Zink spaces}
\label{review}
\subsection{Rapoport-Zink spaces}
\subsubsection{Unitary Rapoport-Zink space $\mathcal{N}^{u}$}
\label{rz-unitary}
Let $E$ be the unramified quadratic extension of $F$, and $\mathcal{O}_{E}$ be the integer ring of $E$. Fix an $\mathcal{O}_F$-algebra embedding $\phi_{0} : \mathcal{O}_{E}\rightarrow \ofb$ and denote by $\phi_{1}$ the embedding $\sigma\circ\phi_{0}: \mathcal{O}_{E}\rightarrow \ofb$. The embedding $\phi_{0}$ induces an embedding between the residue fields $\mathbb{F}_{q^{2}}\rightarrow \mathbb{F}$, which we shall think of as the natural embedding.
\par
Let $n$ be a positive integer. For an $\ofb$-scheme $S$, a hermitian $\mathcal{O}_{E}$-module of signature $(1, n-1)$ over an object $S$ in $\textup{Nilp}_{\ofb}$ is a triple $(X,\iota,\lambda)$ where
\begin{itemize}
    \item[(1)] $X$ is a formal $\varpi$-divisible group over $S$ of relative height $2n$ and dimension $n$;
    \item[(2)] $\iota : \mathcal{O}_{E} \rightarrow \textup{End}(X)$ is an action satisfying the signature $(1, n-1)$ condition, i.e., for $\alpha\in\mathcal{O}_{E}$,
\begin{equation*}
    \textup{char}(\iota(\alpha) : \textup{Lie}\,X)(T) = (T- \phi_{0}(\alpha))(T - \phi_{1}(\alpha))^{n-1} \in \mathcal{O}_{S}[T];
\end{equation*}
\item[(3)] $\lambda : X \rightarrow X^{\vee}$ is a principal polarization such that the associated Rosati involution induces $\alpha \mapsto \sigma(\alpha)$ on $\mathcal{O}_{E}$ via $\iota$.
\end{itemize}
\par
Up to $\mathcal{O}_{E}$-linear quasi-isogeny compatible with the polarizations, there is a unique such triple $\mathbb{X}=(\mathbb{X}, \iota_{\mathbb{X}}, \lambda_{\mathbb{X}})$ over $S=\textup{Spec}\,\mathbb{F}$. The unitary Rapoport-Zink space of signature $(1,n-1)$ is the following functor $\mathcal{N}_{1,n-1}:\textup{Nilp}_{\ofb}\rightarrow\textbf{Sets}$: for a scheme $S$ in the category $\textup{Nilp}_{\ofb}$, the set of $\mathcal{N}_{1,n-1}(S)$ is the isomorphism classes of quadruples $(X, \iota, \lambda, \rho)$, where $(X, \iota, \lambda)$ is a hermitian $\mathcal{O}_{E}$-module of signature $(1,n-1)$ over $S$ and $\rho : X \times_{S} \overline{S} \rightarrow \mathbb{X} \times_{\mathbb{F}} \overline{S}$ is an $\mathcal{O}_{E}$-linear quasi-isogeny of height zero such that $\rho^{\ast}((\lambda_{\mathbb{X}})_{\overline{S}})=\lambda_{\overline{S}}$. Two quadruples $(X, \iota, \lambda, \rho)$ and $(X^{\prime}, \iota^{\prime}, \lambda^{\prime}, \rho^{\prime})$ are isomorphic if there exists an isomorphism $f:X\rightarrow X^{\prime}$ of $\varpi$-divisible groups such that 
\begin{equation*}
    f^{\vee}\circ\lambda^{\prime}\circ f=\lambda,\,\,\rho^{\prime}=f\circ\rho,\,\,\iota^{\prime}(\alpha)\circ f=f\circ\iota(f)\,\,\textup{for all}\,\alpha\in\mathcal{O}_{E}.
\end{equation*}
\begin{theorem}
The functor $\mathcal{N}_{1,n-1}$ is represented by a formal scheme which is formally locally of finite type and formally smooth of relative dimension $n-1$ over $\textup{Spf}\,\ofb$.
\label{smoothu}
\end{theorem}
\begin{proof}
    These facts about $\mathcal{N}_{1,n-1}$ are proved in \cite[Theorem 2.16]{RZ96}, \cite[$\S$2.1]{KR11} and \cite{Vol10}.
\end{proof}
In the following we denote $\mathcal{N}^{u} \coloneqq \mathcal{N}_{1,n-1}$. Let $(\mathbb{Y}, \iota_{\mathbb{Y}}, \lambda_{\mathbb{Y}})$ be the framing object for $n=1$. Let $(\overline{\mathbb{Y}}, \iota_{\overline{\mathbb{Y}}}, \lambda_{\overline{\mathbb{Y}}})=(\mathbb{Y}, \iota_{\mathbb{Y}}\circ\sigma, \lambda_{\mathbb{Y}})$ be its conjugate. There is a canonical lift $(\overline{Y},\iota_{\overline{Y}},\lambda_{\overline{Y}})$ of $(\overline{\mathbb{Y}}, \iota_{\overline{\mathbb{Y}}}, \lambda_{\overline{\mathbb{Y}}})$ over $\ofb$ (cf. \cite[Proposition 2.1]{Gro86}), here $\iota_{\overline{Y}}:\mathcal{O}_{E}\rightarrow\textup{End}_{\ofb}(\overline{Y})$ is an action such that for any $\alpha\in\mathcal{O}_{E}$,
\begin{equation*}
    \textup{char}(\iota(\alpha) : \textup{Lie}\,\overline{Y})(T) = T - \phi_{1}(\alpha) \in \ofb[T];
\end{equation*}
we fix an isomorphism $\rho_{\overline{Y}}:\overline{Y}\times_{\ofb}\mathbb{F}\rightarrow\overline{\mathbb{Y}}$.
\subsubsection{GSpin Rapoport-Zink space $\mathcal{N}^{o}$}
\label{gspin-rz}
Let $k\geq1$ be an integer. A quadratic lattice of rank $k$ is a pair $(L,q_{L})$ such that $L$ is a free $\mathcal{O}_{F}$-module of rank $k$ and $q_{L}: L\rightarrow F$ is a quadratic form on $L$. The quadratic form $q_{L}$ induces a symmetric bilinear form $(\cdot,\cdot):L\times L\rightarrow F$ by $(x,y)=\frac{1}{2}(q(x+y)-q(x)-q(y))$. Let $L^{\sharp}=\{x\in L\otimes_{\mathcal{O}_{F}}F : (x,L)\subset \mathcal{O}_{F}\}$. We say a quadratic lattice is integral if $q_{L}(x)\in\mathcal{O}_{F}$ for all $x\in L$, is self-dual if it is integral and $L$ = $L^{\sharp}$. In this article we will only work with quadratic lattice over $\mathbb{Z}_p$.
\par
Let $n\geq2$ be an integer, and $V$ be a self-dual quadratic lattice of rank $m=n+1$ over $\mathbb{Z}_p$. The Clifford algebra $C(V)$ of $V$ is the quotient of tensor algebra $T^{\bullet}(V)$ by the relations $v\cdot v=q(v)\cdot1$ for all $v\in V$. It is an associated $\mathbb{Z}/2\mathbb{Z}$-graded algebra with grading given by the parity of the tensors. Let $C^{+}(V)$ be the even part, while $C^{-}(V)$ be the odd part. The canonical involution on $C(V)$ is characterized by $(v_1\cdots v_d)^{\ast}=v_d\cdots v_1$ for all $v_1,\cdots, v_d\in V$. We fix an element $\delta\in C(V)^{\times}$ with $\delta^{\ast}=-\delta$, then
\begin{equation}
    \psi_{\delta}(c_1,c_2)=\textup{Trd}(c_1\delta c_2^{\ast})
    \label{polarization}
\end{equation}
is a perfect symplectic form on $C(V)$.
\par
The spinor similitude group $G = \textup{GSpin}(V)$ is a reductive group over $\mathbb{Z}_p$. For any $\mathbb{Z}_p$-algebra $R$, let $V_R=V\otimes_{\mathbb{Z}_p}R$ and $C^{+}(V_{R})=C^{+}(V)\otimes_{\mathbb{Z}_p}R$, the set of $R$-points of $G$ is given by
\begin{equation*}
    G(R)=\{g\in C^{+}(V_{R})^{\times}: g V_{R}g^{-1}=V_{R}, g^{\ast}g\in R^{\times}\}.
\end{equation*}
The group $G$ acts on $C(V)$ by left multiplication, yielding a closed immersion $G\rightarrow\textup{GL}(C(V))$. Let $(s_{\alpha})_{\alpha\in I}$ be a finite set of tensors $s_\alpha$ in the total tensor algebra $C(V)^{\otimes}$ cutting out $G$ from $\textup{GL}(C(V))$. One of the tensors $\boldsymbol{\pi}$ can be chosen as follows (\cite[$\S$1.3]{Mad16}): The left multiplication action of $V$ on $C(V)$ gives a $G$-equivariant embedding $V\hookrightarrow \textup{End}(C(V))$. We identify $V\subset\textup{End}(C(V))$ via this embedding. Then $f\circ f=q(f)\cdot\textup{id}_{C}$ for all $f\in V$. The non-degenerate symmetric bilinear form $(\cdot,\cdot)$ on $\textup{End}(C(V))$ defined by
\begin{equation*}
    (f_1,f_2)\coloneqq 2^{-m}\textup{tr}(f_1\circ f_2)
\end{equation*}
extends the symmetric bilinear form on $V$. Let $\{e_i\}_{i=1}^{m}$ be an orthogonal $\mathbb{Z}_p$-basis of $V$, define
\begin{equation}
    \bpi:\textup{End}(C(V))\rightarrow\textup{End}(C(V)),\,\,\bpi(f)\coloneqq\sum\limits_{i=1}^{m}\frac{(f,e_i)}{(e_i,e_i)}e_i.
    \label{pro-local}
\end{equation}
\par
Fix a basis $\{x_i\}_{i=1}^{m}$ of $V$ for which the inner product matrix of $V$ has the form 
\begin{equation*}
    \left((x_i,x_j)\right)=\begin{pmatrix}
        0 & 1 &    &      &      \\
        1 & 0 &    &      &      \\
          &   &\ast&      &      \\
          &   &    &\ddots&      \\
          &   &    &      &\ast
    \end{pmatrix}.
\end{equation*}
This choice of basis determines a cocharacter $\mu:\mathbb{G}_{m}\rightarrow G$ (cf. \cite[$\S$4.2.1]{HP17}) by
\begin{equation*}
    \mu(t)=t^{-1}x_1x_2+x_2x_1.
\end{equation*}
We also set $b=x_3(p^{-1}x_1+x_2)\in G(\mathbb{Q}_{p})$.
\par
Howard and Pappas \cite[Proposition 4.2.6]{HP17} proved that the quadruple $(G, b, \mu, C(V))$ is a local unramified Shimura-Hodge datum. Let $\Breve{\mathbb{Q}}_p$ be the completion of the maximal unramified extension of $\mathbb{Q}_p$, with ring of integers $\Breve{\mathbb{Z}}_p$. Set $C^{\vee} = \textup{Hom}_{\mathbb{Z}_{p}}(C(V), \mathbb{Z}_p)$ with the contragredient action of $G$. By \cite[Lemma 2.2.5]{HP17}, this local unramified Shimura-Hodge data gives rise to a (unique up to isomorphism) supersingular $p$-divisible group $\mathbb{X}_{V}$ over $\mathbb{F}$ whose contravariant Dieudonne module $\mathbb{D}(\mathbb{X}_{V})(\Breve{\mathbb{Z}}_p)$ is given by $C^{\vee}_{\Breve{\mathbb{Z}}_p}\coloneqq C^{\vee}\otimes_{\mathbb{Z}_{p}}\Breve{\mathbb{Z}}_p$ with Frobenius $\mathbf{F} = b \circ \sigma$, we also obtain $t_{\alpha,0}=s_{\alpha}\otimes1\in(C^{\vee\otimes})_{\zpb}=\mathbb{D}(\mathbb{X}_{V})^{\otimes}$, which are $\mathbf{F}$-invariant. Moreover, the perfect symplectic form $\psi_\delta$ in (\ref{polarization}) on $C(V)$ determines a principal polarization $\lambda_{0}:\mathbb{X}_{V}\rightarrow\mathbb{X}_{V}^{\vee}$.
\par
Associated to the local unramified Shimura-Hodge data $(G, b, \mu, C(V))$, we have a GSpin Rapoport–Zink space $\textup{RZ}(V) = \textup{RZ}(G, b, \mu, C(V))$ of Hodge type (\cite[$\S$4.3]{HP17}, see also \cite{Kim18}) parametrizing $p$-divisible groups with crystalline Tate tensors. More precisely, it is a formal scheme over $\textup{Spf}\,\zpb$ representing the functor sending $R\in\textup{ANilp}_{\zpb}^{\textup{fsm}}$ to the set of isomorphism classes of tuples $(X,(t_{\alpha})_{\alpha\in I},\rho)$, where
\begin{itemize}
    \item[$\bullet$] $X$ is a $p$-divisible group over $\textup{Spec}\,R$.
    \item[$\bullet$] $(t_{\alpha})_{\alpha\in I}$ is a collection of crystalline Tate tensors of $X$.
    \item[$\bullet$] $\rho:\mathbb{X}_{V}\times_{\mathbb{F}}R/J\rightarrow X\times_{R}R/J$ is a framing, i.e., a quasi-isogeny such that each $t_{\alpha}$ pulls back to $t_{\alpha,0}$ under $\rho$, where $J$ is some ideal of definition of $R$ such that $p\in J$.
\end{itemize}
The tuple $(X,(t_{\alpha})_{\alpha\in I},\rho)$ is required to satisfy additional assumptions \cite[Definition 2.3.3 (ii),(iii)]{HP17}.
\begin{theorem}
The GSpin Rapoport-Zink space $\textup{RZ}(V)$ is formally locally of finite type and formally smooth of relative dimension $n - 1$ over $\textup{Spf}\,\zpb$.
\label{smootho}
\end{theorem}
\begin{proof}
    This is Theorem A in \cite{HP17}.
\end{proof}
Let $X^{\textup{univ}}$ be the universal $p$-divisible group over $\textup{RZ}(V)$ with the universal quasi-isogeny $\rho^{\textup{univ}}: \mathbb{X}_{V}\times_{\mathbb{F}}\overline{\textup{RZ}(V)}\rightarrow X^{\textup{univ}}\times_{\textup{RZ}(V)}\overline{\textup{RZ}(V)}$. Let $\mathcal{N}^{o}$ be the connected component of $\textup{RZ}(V)$ such that the principal polarization $\lambda_{0}$ of $\mathbb{X}_{V}$ lifts to a principal polarization $\lambda^{\textup{univ}}$ on $X^{\textup{univ}}$ via the universal quasi-isogeny $\rho^{\textup{univ}}$ (cf. \cite[$\S$4.5]{LZ22b}).

\subsection{The space of special quasi-homomorphisms}
\subsubsection{The unitary case}
\label{special-unitary}
Kudla and Rapoport \cite[Definition 3.1]{KR11} defined the space of special homomorphisms to be the $E$-vector space $\mathbb{V}^{u}\coloneqq\textup{Hom}_{\mathcal{O}_{E}} (\overline{\mathbb{Y}}, \mathbb{X})\otimes_{\mathbb{Z}}\mathbb{Q}$. There is a natural $\mathcal{O}_{E}$-valued $\sigma$-Hermitian form on $\mathbb{V}^{u}$ (cf. \cite[(3.1)]{KR11}) given by
\begin{equation*}
    (x,y)\mapsto \lambda_{\overline{\mathbb{Y}}}^{-1}\circ y^{\vee}\circ \lambda_{\mathbb{X}}\circ x\in\textup{End}_{\mathcal{O}_{E}}(\overline{\mathbb{Y}})\stackrel{\sim}\rightarrow\mathcal{O}_{E}.
\end{equation*}
The hermitian space $V^{u}$ is the unique (up to isomorphism) non-degenerate non-split $E/F$-hermitian space of dimension $n$ (cf. \cite[$\S$2.2]{LZ22a}).

\subsubsection{The GSpin case}
\label{special-gspin}
The inclusion $V \subset C(V)^{\textup{op}}$ (where $V$ acts on $C(V)$ via right multiplication) realizes $V \subset \textup{End}_{\mathbb{Z}_p}(C^{\vee})$. Tensoring with $\qpb$ gives a subspace $V_{\qpb} \subset \textup{End}_{\qpb}(C^{\vee}_{\qpb})$. Define the $\sigma$-linear operator $\Phi = \overline{b} \circ \sigma$ on $V_{\qpb}$, where $\overline{b} \in \textup{SO}(V)(\qpb)$ is the image of $b \in G(\qpb)$ under the natural quotient map $G = \textup{GSpin}(V ) \rightarrow \textup{SO}(V )$. Then $(V_{\qpb}, \Phi)$ is an isocrystal. The $\Phi$-fixed vectors form a distinguished $\mathbb{Q}_p$-vector subspace
\begin{equation*}
    \mathbb{V}^{o}\coloneqq V_{\qpb}^{\Phi}\subset \textup{End}^{\circ}(\mathbb{X}) \coloneqq \textup{End}(\mathbb{X})\otimes_{\mathbb{Z}}\mathbb{Q},
\end{equation*}
called the space of special quasi-homomorphisms of $\mathbb{X}_{V}$ (\cite[$\S$4.3.1]{HP17}, see also \cite[$\S$4.2]{LZ22b}). 
\par
Let $\Breve{q}$ be the base change of the quadratic form $q$ on $V$ to $V_{\qpb}$. For any $x\in V_{\qpb}$, we have $x\circ x=\Breve{q}(x)\cdot\textup{id}_{\mathbb{X}_{V}}$ in $\textup{End}_{\qpb}(C^{\vee}_{\qpb})$ by the construction of the Clifford algebra $C(V)$. Let $q_{\mathbb{V}^{o}}$ be the restriction of the quadratic form $\Breve{q}$ on $V_{\qpb}$ to $\mathbb{V}^{o}$. Therefore $x \circ x = q_{\mathbb{V}^{o}}(x)\cdot \textup{id}_{\mathbb{X}}$ for $x\in\mathbb{V}^{o}$. The quadratic form $q_{\mathbb{V}^{o}}$ values in $\mathbb{Q}_p$ and the quadratic space $\mathbb{V}^{o}$ is the unique (up to isomorphism of quadratic spaces) non-degenerate quadratic space with the same dimension and discriminant, but opposite Hasse invariant as $V$ (cf. \cite[Proposition 4.2.5, $\S$4.3.1]{HP17}).

\subsection{Special cycles and difference divisors}
\label{special-cycles}
\begin{definition}
In the unitary case, for any subset $L \subset \mathbb{V}^{u}$, define the special cycle $\mathcal{Z}(L)$ in $\mathcal{N}^{u}$ as the following subfunctor: for any object $S$ in $\textup{Nilp}_{\ofb}$, the set $\mathcal{Z}(L)(S)$ consists of elements $(X,\iota_{X},\lambda_{X},\rho_{X})\in\mathcal{N}^{u}(S)$ such that the quasi-homomorphism $\rho_{X}^{-1}\circ x\circ \rho_{\overline{Y}}: \overline{Y}\times_{S}\overline{S}\rightarrow X\times_{S}\overline{S}$ extends to a homomorphism from $\overline{Y}$ to $X$ for all $x\in L$.
\par
In the GSpin case, for any subset $L \subset \mathbb{V}^{o}$, define the special cycle $\mathcal{Z}(L) \subset \mathcal{N}^{o}$ to be the closed formal subscheme cut out by the condition
\begin{equation*}
    \rho^{\textup{univ}}\circ x\circ (\rho^{\textup{univ}})^{-1}\subset\textup{End}(X^{\textup{univ}}),
\end{equation*}
for all $x \in L$.
\label{cycle}
\end{definition}
\begin{remark}
    We use the same symbol $\mathcal{Z}(L)$ for special cycles, although they may lie in different formal schemes. In the following discussion, we will always make clear on which space the special cycle lie in.
\end{remark}
\begin{proposition}
Let $x\in\mathbb{V}^{u}$ (resp. $x\in\mathbb{V}^{o}$) be a nonzero element such that $(x,x)\in\mathcal{O}_{E}$ (resp. $q_{\mathbb{V}^{o}}(x)\in\mathbb{Z}_{p}$), the special cycle $\mathcal{Z}(x)$ is an effective Cartier divisor in $\mathcal{N}^{u}$ (resp. $\mathcal{N}^{o}$) and flat over $W$.
\label{divisor}
\end{proposition}
\begin{proof}
    For the unitary case, this is proved in \cite[Proposition 3.5, Lemma 3.7]{KR11}. For the GSpin case, this is proved in \cite[Proposition 4.10.1]{LZ22b}
\end{proof}
\begin{definition}
    Let $x\in\mathbb{V}^{u}$ (resp. $x\in\mathbb{V}^{o}$) be a nonzero element such that $(x,x)\in\mathcal{O}_{E}$ (resp. $q_{\mathbb{V}^{o}}(x)\in\mathbb{Z}_{p}$). For a point $z\in\mathcal{N}^{u}(\mathbb{F})$ (resp. $\mathcal{N}^{o}(\mathbb{F})$), let $f_{x,z}\in \widehat{\mathcal{O}}_{\mathcal{N}^{u},z}$ (resp. $\in\widehat{\mathcal{O}}_{\mathcal{N}^{o},z}$) be the local equation of the closed formal subscheme $\mathcal{Z}(x)$ at $z$, i.e., $\widehat{\mathcal{O}}_{\mathcal{Z}(x),z}\simeq \widehat{\mathcal{O}}_{\mathcal{N}^{u},z}/(f_{x,z})$ (resp. $\widehat{\mathcal{O}}_{\mathcal{Z}(x),z}\simeq \widehat{\mathcal{O}}_{\mathcal{N}^{o},z}/(f_{x,z})$). Notice that the element $f_{x,z}$ is well-defined up to units in $\widehat{\mathcal{O}}_{\mathcal{N}^{u},z}$ (resp. $\widehat{\mathcal{O}}_{\mathcal{N}^{o},z}$).
\end{definition}
\par
Let $x\in\mathbb{V}^{u}$ ($x\in\mathbb{V}^{o}$) be a nonzero element such that $(x,x)\in\mathcal{O}_{E}$ (resp. $q_{\mathbb{V}^{o}}(x)\in\mathbb{Z}_{p}$), we have a natural closed immersion of closed formal schemes $\mathcal{Z}(\varpi^{-1}x)\hookrightarrow\mathcal{Z}(x)$ (resp. $\mathcal{Z}(p^{-1}x)\hookrightarrow\mathcal{Z}(x)$) by Definition \ref{cycle}. The closed immersion $\mathcal{Z}(\varpi^{-1}x)\hookrightarrow\mathcal{Z}(x)$ (resp. $\mathcal{Z}(p^{-1}x)\hookrightarrow\mathcal{Z}(x)$) implies the divisibility $f_{\varpi^{-1}x,z}\vert f_{x,z}$ at every point $z\in\mathcal{N}^{u}(\mathbb{F})$ (resp. $\mathcal{N}^{o}(\mathbb{F})$).
\begin{definition}
    Let $x \in \mathbb{V}^{u}$ (resp. $x\in\mathbb{V}^{o}$) be an element such that $(x,x)\in\mathcal{O}_{E}$ (resp. $q_{\mathbb{V}^{o}}(x)\in\mathbb{Z}_{p}$). Define the difference divisor associated to $x$ to be the following Cartier divisor on $\mathcal{N}^{u}$ (resp. $\mathcal{N}^{o}$),
    \begin{equation*}
        \mathcal{D}(x)\coloneqq\mathcal{Z}(x)-\mathcal{Z}(\varpi^{-1}x)\,\,\,\,\textup{(resp. $\mathcal{D}(x)\coloneqq\mathcal{Z}(x)-\mathcal{Z}(p^{-1}x)$)}.
    \end{equation*}
    i.e., at a point $z\in\mathcal{Z}(x)(\mathbb{F})$, the difference divisor $\mathcal{D}(x)$ is cut out by $d_{x,z}\coloneqq f_{x,z}/f_{\varpi^{-1}x,z}\in\widehat{\mathcal{O}}_{\mathcal{N}^{u},z}$ (resp. $\widehat{\mathcal{O}}_{\mathcal{N}^{o},z}$).
    \label{difference}
\end{definition}

\section{Deformation theory}
\label{def}
\subsection{The unitary case}
\subsubsection{Preliminaries on linear algebra}
\begin{definition}
    Let $R$ be a commutative ring. Let $D$ be a locally free $R$-module of finite rank $n$, a hyperplane $P$ in $D$ is a locally free $R$-submodule of $D$ of rank $n-1$, a line $L$ in $D$ is a locally free $R$-submodule of $D$ of rank $1$.
    \label{plane-and-line}
\end{definition}
\begin{definition}
    Let $R$ be a commutative ring. Let $D_1$ and $D_2$ be two $R$-modules. An $R$-bilinear pairing $\langle\cdot,\cdot\rangle:D_1\times D_2\rightarrow R$ is perfect if the induced maps $D_1\rightarrow D_2^{\vee},\,d_1\rightarrow\langle d_1,\cdot\rangle$ and $D_2\rightarrow D_1^{\vee},\,d_1\rightarrow\langle d_1,\cdot\rangle$ are isomorphisms.
\end{definition}
\begin{lemma}
    Let $R$ be a commutative ring. Let $D_1$ and $D_2$ be two locally free $R$-module of finite rank $n$. Let $\langle\cdot,\cdot\rangle:D_1\times D_2\rightarrow R$ be an $R$-bilinear pairing. Then we have a bijection,
    \begin{align*}
        \{\textup{Lines in}\,\,D_1.\}&\longleftrightarrow\{\textup{Hyperplanes in $D_2$}.\},\\
        L & \longmapsto P=\{x\in D:\,\langle l,x\rangle=0,\,\,\textup{for all}\,\,l\in L.\};\\
        L=\{l\in D_1:\,\langle l,x\rangle=0,\,\,\textup{for all}\,\,x\in P.\}.&\testleftlong P
    \end{align*}
    \label{corre}
\end{lemma}
\begin{proof}
    This can be proved by purely linear algebra.
\end{proof}
\subsubsection{Deformation theory}
\label{def-unitary}
For an object $S\in\textup{Alg}_{\ofb}$ and any $z\in\mathcal{N}^{u}(S)$, the point $z$ corresponds to a hermitian $\mathcal{O}_{E}$-module $(X_z,\iota_{z},\lambda_z,\rho_{z})$ of signature $(1,n-1)$ as in $\S$\ref{rz-unitary}. Let $\mathbb{D}(X_z)$ be the Dieudonne crystal of $X_z$ in the site $\textup{NCRIS}_{\ofb}(S/\textup{Spec}\,\ofb)$, it is a rank $2n$ $\mathcal{O}_{S}^{\textup{crys}}$-module crystal with $\mathcal{O}_E$-action. Under the isomorphism $\mathcal{O}_{E}\otimes_{\mathcal{O}_F}\ofb\simeq\ofb\times\ofb$, we have
\begin{equation*}
    \mathcal{O}_{E}\otimes_{\mathcal{O}_F}\mathcal{O}_{S}^{\textup{crys}}\simeq\mathcal{O}_{S}^{\textup{crys}}\times \mathcal{O}_{S}^{\textup{crys}}.
\end{equation*}
Accordingly, the crystal $\mathbb{D}(X_z)$ has the following decomposition
\begin{equation*}
    \mathbb{D}(X_z)=\mathbb{D}_0(X_z)\oplus\mathbb{D}_1(X_z),
\end{equation*}
where both of $\mathbb{D}_0(X_z)$ and $\mathbb{D}_1(X_z)$ are $\mathcal{O}_{S}^{\textup{crys}}$-modules of rank $n$, and $\mathcal{O}_E$ acts on $\mathbb{D}_i(X_z)$ by multiplication through $\phi_i:\mathcal{O}_{E}\rightarrow\ofb$ for $i=0,1$.
\par
For simplicity, let $\mathbf{D}_z=\mathbb{D}_0(X_z)$ and $\overline{\mathbf{D}}_z=\mathbb{D}_1(X_z)$. Let $0\subset\textup{Fil}^{1}\mathbb{D}(X_z)(S)\subset\mathbb{D}(X_z)(S)$ be the Hodge filtration. Let $\textup{Fil}^{1}\mathbf{D}_z(S)=\textup{Fil}^{1}\mathbb{D}(X_z)(S)\cap\mathbf{D}_z(S)$ and $\textup{Fil}^{1}\overline{\mathbf{D}}_z(S)=\textup{Fil}^{1}\mathbb{D}(X_z)(S)\cap\overline{\mathbf{D}}_z(S)$. The compatibility  with the $\mathcal{O}_{E}$-action implies that
\begin{equation*}
    \textup{Fil}^{1}\mathbb{D}(X_z)(S)=\textup{Fil}^{1}\mathbf{D}_z(S)\oplus\textup{Fil}^{1}\overline{\mathbf{D}}_z(S).
\end{equation*}
\par
Recall that there is a principal polarization $\lambda_z:X_z\rightarrow X_z^{\vee}$. It induces an isomorphism of crystals $\mathbb{D}(\lambda_z):\mathbb{D}(X_z)\rightarrow\mathbb{D}(X_z^{\vee})$. Notice that $\mathbb{D}(X_z^{\vee})\simeq\mathbb{D}(X_z)^{\vee}$ by \cite[Theorem 5.1.8]{BBM82}. Hence $\mathbb{D}(\lambda_z)$ induces a perfect symplectic pairing between crystals $\langle\cdot,\cdot\rangle:\mathbb{D}(X_z)\times\mathbb{D}(X_z)\rightarrow\mathcal{O}_{S}^{\textup{crys}}$ (cf. \cite[Chapter 2]{Jong93}). Moreover, the Hodge filtration $\textup{Fil}^{1}\mathbb{D}(X_z)(S)$ is totally isotropic under this symplectic pairing (cf. \cite[Proposition 5.1.10]{BBM82}).
\begin{lemma}
Let $R\rightarrow S$ be a surjection in $\textup{Alg}_{\ofb}$ whose kernel admits divided powers. The perfect symplectic pairing defined by the principal polarization $\lambda_z$ induces a perfect pairing 
    \begin{equation*}
        \langle\cdot,\cdot\rangle:\mathbf{D}_z(R)\times\overline{\mathbf{D}}_z(R)\rightarrow R.
    \end{equation*}
 Moreover, the intersection $\textup{Fil}^{1}\mathbf{D}_z(S)$ is a hyperplane in $\mathbf{D}_z(S)$, $\textup{Fil}^{1}\overline{\mathbf{D}}_z(S)$ is a line in $\overline{\mathbf{D}}_z(S)$, and they are annihilators of each other under the perfect $R$-linear pairing $\langle\cdot,\cdot\rangle$.
\end{lemma}
\begin{proof}
    The principal polarization $\lambda_{z}$ induces a perfect symplectic $R$-linear pairing
    \begin{equation*}
        \langle\cdot,\cdot\rangle:\mathbb{D}(X_z)(R)\times\mathbb{D}(X_z)(R)\rightarrow R.
    \end{equation*}
    By definition, for all elements $a\in\mathcal{O}_{E}$, we have $\langle \mathbb{D}(\iota_z(a))(x),y\rangle=\langle x,\mathbb{D}(\iota_z(\sigma(a)))(y)\rangle$ for all $x,y\in\mathbb{D}(X_z)(R)$. Let $\varepsilon\in\mathcal{O}_{E}^{\times}$ be an element such that $\sigma(\varepsilon)=-\varepsilon$, we get  
    \begin{align*}
        \langle \varepsilon x,y\rangle&=\langle x,-\varepsilon y\rangle,\,\,\textup{for all}\,\,x,y\in\mathbf{D}_z(R);\\
        \langle -\varepsilon x,y\rangle&=\langle x,\varepsilon y\rangle,\,\,\textup{for all}\,\,x,y\in\overline{\mathbf{D}}_z(R).
    \end{align*}
    Hence $\langle  x,y\rangle=0$ if $x,y\in\mathbf{D}_z(R)$ or $x,y\in\overline{\mathbf{D}}_z(R)$. Therefore the perfect symplectic $R$-linear pairing $\langle\cdot,\cdot\rangle$ restricts to a perfect $R$-linear pairing $\langle\cdot,\cdot\rangle:\mathbf{D}_z(R)\times\overline{\mathbf{D}}_z(R)\rightarrow R$.
    \par
    For the second claim, notice that there is an exact sequence of locally free $S$-modules,
    \begin{equation*}
        0\longrightarrow\textup{Fil}^{1}\mathbb{D}(X_z)(S)\longrightarrow\mathbb{D}(X_z)(S)\longrightarrow\textup{Lie}\,X_z\longrightarrow0.
    \end{equation*}
    The signature $(1,n-1)$ condition on $\textup{Lie}\,X_z$ implies that the intersection $\textup{Fil}^{1}\mathbf{D}_z(S)$ is a hyperplane in $\mathbf{D}_z(S)$, $\textup{Fil}^{1}\overline{\mathbf{D}}_z(S)$ is a line in $\overline{\mathbf{D}}_z(S)$. We also know that $\textup{Fil}^{1}\mathbf{D}_z(S)$ is totally isotropic under the pairing $\langle\cdot,\cdot\rangle$, hence $\textup{Fil}^{1}\mathbf{D}_z(S)$ and $\textup{Fil}^{1}\overline{\mathbf{D}}_z(S)$ are annihilators of each other.
\end{proof}
Let $\mathbb{D}(\overline{\mathbb{Y}})$ be the (covariant) Dieudonne module of $\overline{\mathbb{Y}}$ over $\ofb$, it is a free $\ofb$-module of rank $2$. We have a decomposition $\mathbb{D}(\overline{\mathbb{Y}})=\mathbb{D}_0(\overline{\mathbb{Y}})\oplus\mathbb{D}_1(\overline{\mathbb{Y}})$, both $\mathbb{D}_0(\overline{\mathbb{Y}})$ and $\mathbb{D}_1(\overline{\mathbb{Y}})$ are free $\ofb$-module of rank $1$ by the signature condition. Let $\overline{1}_0$ be an $\ofb$-generator of $\mathbb{D}_0(\overline{\mathbb{Y}})$. Let $\overline{1}_1=\varpi^{-1}\Phi\overline{1}_0$ where $\Phi$ is the Frobenius morphism on the Dieudonne module $\mathbb{D}(\overline{\mathbb{Y}})$. We have $\mathbb{D}_1(\overline{\mathbb{Y}})=\ofb\cdot\overline{1}_1$ (\cite[Remark 2.5]{KR11}).
\par
For a point $z\in\mathcal{N}^{u}(\mathbb{F})$. The (covariant) Dieudonne module $\mathbb{D}(X_{z})$ of $X_{z}$ also has a decomposition $\mathbb{D}_0(X_{z})\oplus\mathbb{D}_1(X_{z})$ as $\ofb$-module. Let $\mathbf{D}_z=\mathbb{D}_0(X_{z})$ and $\overline{\mathbf{D}}_z=\mathbb{D}_1(X_{z})$. Let $x\in\mathbb{V}^{u}$ be a special quasi-homomorphism, it induces an $\Breve{F}$-linear homomorphism $\mathbb{D}(x):\mathbb{D}(\overline{\mathbb{Y}})\otimes_{\ofb}\Breve{F}\rightarrow\mathbb{D}(X_{z})\otimes_{\ofb}\Breve{F}$, sending $\mathbb{D}_i(\overline{\mathbb{Y}})\otimes_{\ofb}\Breve{F}$ to $\mathbb{D}_i(X_{z})\otimes_{\ofb}\Breve{F}$ for $i=0,1$. Define
\begin{equation}
    i_{\textup{crys},z}: \mathbb{V}^{u}\longrightarrow\mathbf{D}_{z}\otimes_{\ofb}\Breve{F},\,\,
    x\longmapsto x_{\textup{crys},z}\coloneqq\mathbb{D}(x)(\overline{1}_0)\in\mathbf{D}_{z}\otimes_{\ofb}\Breve{F}.
    \label{crystal-unitary}
\end{equation}
\par
More generally, for $S\in\textup{Alg}_{\ofb}$, let $z\in\mathcal{Z}(L)(S)$. For an arbitrary element $x\in L$, there is an isogeny $\overline{Y}_{S}\rightarrow X_z$, where $\overline{Y}$ is the canonical lift of $\overline{Y}$ and $\overline{Y}_S=\overline{Y}\times_{\ofb}S$, lifting the element $x\in\mathbb{V}^{u}=\textup{Hom}^{\circ}_{\mathcal{O}_E}(\overline{\mathbb{Y}},\mathbb{X})$.   Let $\mathbb{D}(\overline{Y}_S)$ be the Dieudonne crystal of $\overline{Y}_S$, then $\mathbb{D}(\overline{Y}_S)\simeq\mathbb{D}(\overline{\mathbb{Y}})\otimes_{\ofb}\mathcal{O}_{S}^{\textup{crys}}$. Similarly, we have a decomposition $\mathbb{D}(\overline{Y}_S)=\mathbb{D}_0(\overline{Y}_S)\oplus\mathbb{D}_1(\overline{Y}_S)$, here $\mathbb{D}_i(\overline{Y}_S)\simeq\mathbb{D}_i(\overline{\mathbb{Y}})\otimes_{\ofb}\mathcal{O}_{S}^{\textup{crys}}$ for $i=0,1$, then the element $\overline{1}_0$ we defined above is also a generator of the $\mathcal{O}_{S}^{\textup{crys}}$-module $\mathbb{D}_0(\overline{Y}_S)$.
\par
Moreover, the Hodge filtration on $\mathbb{D}(\overline{Y}_S)$ equals to $0\subset\mathbb{D}_0(\overline{Y}_S)\subset\mathbb{D}(\overline{Y}_S)$. Let $R\rightarrow S$ be a surjection in $\textup{Alg}_{\ofb}$ whose kernel admits a nilpotent $\mathcal{O}_F$-pd-structure, the element $x$ induces a morphism of crystals $\mathbb{D}(x):\mathbb{D}(\overline{Y}_S)\rightarrow\mathbb{D}(X_z)$. Let $x_{\textup{crys},z}(R)=\mathbb{D}(x)(\overline{1}_0)\in\mathbf{D}_z(R)$.
\begin{lemma}
Let $R\rightarrow S$ be a surjection in $\textup{Alg}_{\ofb}$ whose kernel admits nilpotent $\mathcal{O}_{F}$-pd-structure.\\
    (i)\,Let $z_0\in\mathcal{N}^{u}(\mathbb{F})$ and $\widehat{\mathcal{N}}^{u}_{z_0}$ be the completion of $\mathcal{N}^{u}$ at $z_0$. Let $z\in\widehat{\mathcal{N}}^{u}_{z_0}(S)$. Then there is a natural bijection
\begin{equation*}
    \left\{\textup{Lifts}\,\, \tilde{z}\in\widehat{\mathcal{N}}^{u}_{z_0}(R)\,\,\textup{of}\,\,z. \right\}\stackrel{\sim}\longrightarrow\left\{\textup{Lines in $\overline{\mathbf{D}}_{z}(R)$}\,\,\textup{lifting the line $\textup{Fil}^{1}\overline{\mathbf{D}}_z(S)$}\right\}.
\end{equation*}
which maps a lift $\Tilde{z}\in\widehat{\mathcal{N}}^{u}_{z_0}(R)$ to the line $\textup{Fil}^{1}\overline{\mathbf{D}}_{\Tilde{z}}(R)\subset\overline{\mathbf{D}}_{\Tilde{z}}(R)\simeq\overline{\mathbf{D}}_z(R)$.
    \\
    (ii)\,Let $L\subset\mathbb{V}^{u}$ be a $\mathcal{O}_{F}$-lattice of rank $r\geq1$. Let $z_0\in\mathcal{Z}(L)(\mathbb{F})$ and $\widehat{\mathcal{Z}(L)}_{z_0}$ be the completion of $\mathcal{Z}(L)$ at $z_0$, then there is a natural bijection compatible with the bijection in (i)
\begin{equation*}
      \left\{\textup{Lifts}\,\, \Tilde{z}\in\widehat{\mathcal{Z}(L)}_{z_0}(R)\,\,\textup{of}\,\,z. \right\}\stackrel{\sim}\longrightarrow 
    \left\{\begin{array}{ll}
     \textup{Lines in $\overline{\mathbf{D}}(R)$ lifting $\textup{Fil}^{1}\overline{\mathbf{D}}_z(S)$ and}\\
    \textup{orthogonal to $x_{\textup{crys},z}(R)$ for all $x\in L$.}\\
    \textup{under the pairing $\langle\cdot,\cdot\rangle$.}
    \end{array}\right\}.
\end{equation*}
\label{deformu}
\end{lemma}
\begin{proof}
    We first prove (i). We construct the reverse of the map stated in the lemma. Let $L\subset\overline{\mathbf{D}}_{z}(R)$ be a line lifting the line $\textup{Fil}^{1}\overline{\mathbf{D}}_z(S)$. Let $P\subset\mathbf{D}_z(R)$ be the annihilator of $L$ under the perfect pairing $\langle\cdot,\cdot\rangle:\mathbb{D}(X_z)(R)\times\mathbb{D}(X_z)(R)\rightarrow R$. The $R$-module $P$ is a hyperplane by Lemma \ref{plane-and-line}, lifting the hyperplane $\textup{Fil}^{1}\mathbf{D}_z(S)$. Let $H=L\oplus P\subset\mathbb{D}(X_z)(R)$, it is a locally free $S$-module of rank $n$ lifting the Hodge filtration $\textup{Fil}^{1}\mathbb{D}(X_z)(S)$. By Grothendieck-Messing deformation theory, it corresponds to a $\varpi$-divisible group $\Tilde{X}_z$ over $R$ lifting $X_z$. Notice that the Hodge filtration $0\subset H\subset\textup{Fil}^{1}\mathbb{D}(X_z)(R)$ is stable under $\mathcal{O}_{E}$-action, hence $\iota_{z}$ lifts to $\Tilde{\iota}_z:\mathcal{O}_{E}\rightarrow\textup{End}(\Tilde{X}_z)$. The signature $(1,n-1)$ condition on $\textup{Lie}\,\tilde{X}_z$ is automatic by construction. Hence we construct a point $\Tilde{z}\in\widehat{\mathcal{N}}^{u}_{z_0}(R)$ which lifts $z$. It's easy to see that this map is the reverse of the map stated in the lemma.
    \par
    Now we prove (ii). Let $\Tilde{z}\in\widehat{\mathcal{N}}^{u}_{z_0}(R)$ be a lift of the point $z$. If $\Tilde{z}\in\widehat{\mathcal{Z}(L)}_{z_0}(R)$. Let $x\in L$ be an arbitrary element. The morphism $\mathbb{D}(x)(R)$ preserves the Hodge filtrations, i.e., it maps $\mathbb{D}_0(\overline{Y}_R)(R)\simeq\mathbb{D}_0(\overline{Y}_S)(R)$ to $\textup{Fil}^{1}\mathbb{D}(X_{\tilde{z}})(R)$. Hence $x_{\textup{crys},z}(R)=D(x)(R)(\overline{1}_0)\in\textup{Fil}^{1}\mathbb{D}(X_{\tilde{z}})(R)\cap\mathbf{D}_{\tilde{z}}(R)=\textup{Fil}^{1}\mathbf{D}_{\Tilde{z}}(R)$. Notice that the hyperplane $\textup{Fil}^{1}\mathbf{D}_{\Tilde{z}}(R)$ is the annihilator of the line $\textup{Fil}^{1}\overline{\mathbf{D}}_{\Tilde{z}}(R)$. Therefore the line corresponding to $\Tilde{z}$ under the bijection in (i) is orthogonal to $x_{\textup{crys},z}(R)$ for all $x\in L$. 
    \par
    If we have a line $L\subset\overline{\mathbf{D}}_{z}(R)$ lifting the line $\textup{Fil}^{1}\overline{\mathbf{D}}_z(S)$ and orthogonal to $x_{\textup{crys},z}(R)$ for all $x\in L$. Let $P\subset\mathbf{D}_z(R)$ be the annihilator of $L$ under the perfect pairing. The orthogonal condition is equivalent to $x_{\textup{crys},z}(R)\in P$ for all $x\in L$. Let $\Tilde{z}\in\widehat{\mathcal{N}}^{u}_{z_0}(R)$ be the point corresponding to the line $L$. The Hodge filtration of $X_{\Tilde{z}}$ is given by $0\subset L\oplus P\subset\mathbb{D}(X_z)(R)$. Let $x\in L$ be an element, the morphism $\mathbb{D}(x)(R):\mathbb{D}(\overline{Y}_R)(R)\rightarrow\mathbb{D}(X_{\Tilde{z}})(R)$ preserves the Hodge filtration since $\mathbb{D}(x)(\overline{1}_0)=x_{\textup{crys},z}(R)\in P\subset\textup{Fil}^{1}\mathbb{D}(X_{\Tilde{z}})(R)$. Hence $x$ is an isogeny for all $x\in L$, i.e., $\Tilde{z}\in\widehat{\mathcal{Z}(L)}_{z_0}(R)$.
\end{proof}

\subsection{The GSpin case}
\label{cyrso}
\subsubsection{Crystals over $\mathcal{N}^{o}$}
\label{crystal-over-ortho}
There is a rank $m$ $\mathcal{O}_{\mathcal{N}^{o}}^{\textup{crys}}$-module crystal $\mathbf{V}_{\textup{crys}}$ in the site $\textup{NCRIS}_{\zpb}(\mathcal{N}^{o}/\textup{Spf}\,\zpb)$. The construction of $\mathbf{V}_{\textup{crys}}$ is based on the element $\bpi$ in (\ref{pro-local}). By \cite[Theorem 4.9.1]{Kim18}, we obtain from $\bpi$ a universal crystalline Tate tensor $\bpi_{\textup{crys}}$ on the universal $p$-divisible group $X^{\textup{univ}}$ over $\mathcal{N}^{o}$, which induces a projector of crystals
\begin{equation*}
    \bpi_{\textup{crys}}:\textup{End}(\mathbb{D}(X^{\textup{univ}}))\rightarrow\textup{End}(\mathbb{D}(X^{\textup{univ}})).
\end{equation*}
Let $\vcrys$ be the image of this projector, it is a rank $m$ $\mathcal{O}_{\mathcal{N}^{o}}^{\textup{crys}}$-module crystal (\cite[$\S$4.3]{LZ22b}).
\par
For any $S\in\textup{Alg}_{\zpb}$ and any $z\in\mathcal{N}^{o}(S)$, we similarly have a projector of crystals
\begin{equation*}
    \bpi_{\textup{crys},z}:\textup{End}(\mathbb{D}(X_z))\rightarrow\textup{End}(\mathbb{D}(X_z)).
\end{equation*}
whose image $\vcrysz := \textup{im}(\bpi_{\textup{crys},z})$ is a crystal of $\mathcal{O}_{S}^{\textup{crys}}$-modules of rank $m$. Here $X_z$ denotes the $p$-divisible group over $S$ obtained by the base change of $X^{\textup{univ}}$ to $z$.
\par
Let $R\rightarrow S$ be a surjection in $\textup{Alg}_{\zpb}$ whose kernel admits nilpotent divided powers, we have $\vcrysz(R)$ a projective $R$-module of rank $m$. It is equipped with a non-degenerate symmetric $R$-bilinear form $(\cdot,\cdot)$. The projective $S$-module $\vcrysz(S)$ is equipped with a Hodge filtration $\textup{Fil}^{1}\vcrysz(S)$, which is an isotropic $S$-line (\cite[$\S$4.3]{LZ22b}). The $S$-line $\textup{Fil}^{1}\vcrysz(S)$ is isotropic because it is contained in $\textup{Fil}^{1}\textup{End}(\mathbb{D}(X_z))$ which consists of endomorphisms mapping the Hodge filtration $\textup{Fil}^{1}\mathbb{D}(X_z)$ to $0$, and $\mathbb{D}(X_z)$ to $\textup{Fil}^{1}\mathbb{D}(X_z)$. Therefore the composition $v\circ v=0$ for every element $v\in \textup{Fil}^{1}\vcrysz(S)$.
\par
Moreover, let $L\subset\mathbb{V}^{o}$ be a subset, for $z\in\mathcal{Z}(L)(S)$ and an arbitrary element $x\in L$, the crystalline realization $x_{\textup{crys},z}(R)\in\textup{End}(\mathbb{D}(X_z))(R)$ lies in the image of $\bpi_{\textup{crys},z}(R)$, hence it is an element in $\vcrysz(R)$ (cf. \cite[$\S$4.6]{LZ22b}).
\begin{lemma}
    Let $R\rightarrow S$ be a surjection in $\textup{Art}_{\zpb}$ whose kernel admits nilpotent divided powers.\\
    (i)\,Let $z_0\in\mathcal{N}^{o}(\mathbb{F})$ and $\widehat{\mathcal{N}}^{o}_{z_0}$ be the completion of $\mathcal{N}^{o}$ at $z_0$. Let $z\in\widehat{\mathcal{N}}^{o}_{z_0}(S)$. Then there is a natural bijection
    \begin{equation*}
        \left\{\textup{Lifts}\,\, \tilde{z}\in\widehat{\mathcal{N}}^{o}_{z_0}(R)\,\,\textup{of}\,\,z. \right\}\stackrel{\sim}\longrightarrow\left\{\textup{Isotropic $R$-lines in}\,\,\mathbf{V}_{\textup{crys},z}(R)\,\,\textup{lifting}\,\,\textup{Fil}^{1}\mathbf{V}_{\textup{crys},z}(S).\right\},
    \end{equation*}
    which maps a lift $\Tilde{z}\in\widehat{\mathcal{N}}^{o}_{z_0}(R)$ to the Hodge filtration $\textup{Fil}^{1}\mathbf{V}_{\textup{crys},\Tilde{z}}(R)\subset\mathbf{V}_{\textup{crys},\Tilde{z}}(R)\simeq\mathbf{V}_{\textup{crys},z}(R)$.
    \\
    (ii)\,Let $L\subset\mathbb{V}^{o}$ be a $\mathbb{Z}_{p}$-lattice of rank $r\geq1$. Let $z_0\in\mathcal{Z}(L)(\mathbb{F})$ and $\widehat{\mathcal{Z}(L)}_{z_0}$ be the completion of $\mathcal{Z}(L)$ at $z_0$, then there is a natural bijection compatible with the bijection in (i)
\begin{equation*}
      \left\{\textup{Lifts}\,\, \tilde{z}\in\widehat{\mathcal{Z}(L)}_{z_0}(R)\,\,\textup{of}\,\,z. \right\}\stackrel{\sim}\longrightarrow 
    \left\{\begin{array}{ll}
     \textup{Isotropic $R$-lines in $\mathbf{V}_{\textup{crys},z}(R)$ lifting $\textup{Fil}^{1}\vcrysz(S)$}\\
    \textup{and orthogonal to $x_{\textup{crys},z}(R)$ for any $x\in L$.}
    \end{array}\right\}.
\end{equation*}
\label{deformo}
\end{lemma}
\begin{proof}
    This is \cite[Lemma 4.6.2]{LZ22b}.
\end{proof}
Especially, let $z\in\mathcal{N}^{o}(\mathbb{F})$ be a point. Taking successive surjection $\zpb/(p^{n+1})\rightarrow\zpb/(p^{n})\rightarrow\mathbb{F}$, we define $\mathbf{V}_{z}\coloneqq\varprojlim\limits_{n}\mathbf{V}_{\textup{crys},z}(\zpb/p^{n})$, it is a free $\zpb$-module of rank $m$. 
\par
We define a map $i_{\textup{crys},z}:\mathbb{V}^{o}\rightarrow\mathbf{V}_{z}\otimes_{\zpb}\qpb$ in the following way: For an element $x\in \mathbb{V}^{o}$, there exists an integer $k$ such that $p^{k}x$ lifts to an isogeny of $X_{z}$ under the quasi-isogeny $\rho_{z}:\mathbb{X}_{V}\rightarrow X_z$. Let $x^{\prime}=p^{k}x$, then $z\in\mathcal{Z}(x^{\prime})(\mathbb{F})$. Let $x^{\prime}_{\textup{crys},z}=\varprojlim\limits_{n}x^{\prime}_{\textup{crys},z}(\zpb/p^{n})\in\mathbf{V}_{z}$, define 
\begin{equation}
    i_{\textup{crys},z}(x)= p^{-k}x^{\prime}_{\textup{crys},z}.
    \label{crystalline-gspin}
\end{equation}
The definition is independent of the integer $k$ we choose. For any $x\in\mathbb{V}^{o}$, denote by $x_{\textup{crys},z}$ the element $i_{\textup{crys},z}(x)$. In this way we defined a map $i_{\textup{crys},z}:\mathbb{V}^{o}\rightarrow\mathbf{V}_{z}\otimes_{\zpb}\qpb$.
\begin{lemma}
    Let $z\in\mathcal{N}^{o}(\mathbb{F})$ be a point. Let $L\subset\mathbb{V}^{o}$ be a subset. Then $z\in\mathcal{Z}(L)(\mathbb{F})$ if and only if $i_{\textup{crys},z}(L)\subset\mathbf{V}_z$.
    \label{gspin-simple-k-points}
\end{lemma}
\begin{proof}
    It follows from the construction the map $i_{\textup{crys},z}$ that the condition $z\in\mathcal{Z}(L)(\mathbb{F})$ implies $i_{\textup{crys},z}(L)\subset\mathbf{V}_z$. Now we prove the converse direction. Suppose that $i_{\textup{crys},z}(L)\subset\mathbf{V}_z$. Let $x\in L$. We have $x_{\textup{crys},z}\in\mathbf{V}_z\subset\textup{End}_{\zpb}(\mathbb{D}(X_z)_{\zpb})$ where $\mathbb{D}(X_z)_{\zpb}$ is the Dieudonne module of the $p$-divisible group $X_z$. Let $\rho_{z}:\mathbb{X}_{V}\rightarrow X_z$ be the framing quasi-isogeny. Then $x_{\textup{crys},z}=\mathbb{D}(\rho_{z}\circ x\circ\rho_z^{-1})$ by construction. By Dieudonne theory, the quasi-isogeny $\rho_{z}\circ x\circ\rho_z^{-1}$ is indeed an isogeny because the Dieudonne module $\mathbb{D}(X_z)$ is stable under the map $\mathbb{D}(\rho_{z}\circ x\circ\rho_z^{-1})=x_{\textup{crys},z}$. Hence $z\in\mathcal{Z}(x)(\mathbb{F})$ for all $x\in L$. Therefore $z\in\mathcal{Z}(L)(\mathbb{F})$.
\end{proof}
\begin{remark}
    The analogous statement in the unitary case is not true. Let $z\in\mathcal{N}^{u}(\mathbb{F})$ be a point. Let $L\subset\mathbb{V}^{u}$ be a subset. Then $z\in\mathcal{Z}(L)(\mathbb{F})$ implies that $i_{\textup{crys},z}(L)\subset\mathbf{D}_z$ by the construction of the map $i_{\textup{crys},z}$. However, the reverse direction is not true. We refer the reader to Example \ref{difference1} (b). In that case, let $x^{\prime}=\varpi^{-1}x$. We have $x^{\prime}_{\textup{crys},z}\in\mathbf{D}_z$ but $z\notin\mathcal{Z}(x^{\prime})(\mathbb{F})$.
    \par
    In the work of Li and Zhu \cite{LZ17}, they defined a non-degenerate bilinear form $(\cdot,\cdot)_{\mathbf{D}_z}$ on the $\qpb$-vector space $\mathbf{D}_z\otimes_{\zpb}\qpb$ ((2.2.0.3) in loc. cit., see also \cite[$\S$2.1]{KR11}) in the following way:
    \begin{equation*}
        (x,y)_{\mathbf{D}_z}\coloneqq (\varpi\delta)^{-1}\langle x,Fy\rangle\,\,\,\,\textup{for $x,y\in\mathbf{D}_z\otimes_{\zpb}\qpb$},
    \end{equation*}
    where $\delta$ is a fixed element in $\mathcal{O}_{E}^{\times}$ such that $\sigma(\delta)=-\delta$, and $F$ is the Frobenius morphism on the Dieudonne module $\mathbb{D}(X_z)=\mathbf{D}_z\oplus\overline{\mathbf{D}}_z$. Let $\mathbf{D}_z^{\vee}$ be the dual lattice of $\mathbf{D}_z$ under the bilinear form $(\cdot,\cdot)_{\mathbf{D}_z}$. Notice that $\mathbf{D}_z\nsubseteq\mathbf{D}_z^{\vee}$ because there exists an element $x\in\mathbf{D}_z$ such that $(x,x)_{\mathbf{D}_z}\notin\Breve{\mathbb{Z}}_p$. In \cite[Corollary 3.1.4]{LZ17}, they proved that $z\in\mathcal{Z}(L)(\mathbb{F})$ if and only if $i_{\textup{crys},z}(L)\subset\mathbf{D}_z^{\vee}$.
\end{remark}

\subsubsection{Strong divisibility}
\label{std}
Let $z\in\mathcal{N}^{o}(\mathbb{F})$ be a point. By $\S\ref{crystal-over-ortho}$, the $\zpb$-module $\mathbf{V}_z$ is a free $\zpb$-module of rank $m$, with a nondegenerate $\zpb$-bilinear pairing $(\cdot,\cdot)$. The $\qpb$-vector space $\mathbf{V}_{z,\qpb}\coloneqq\mathbf{V}_z\otimes_{\zpb}\qpb$ is equipped with a Frobenius-linear action $\Phi_z:\mathbf{V}_{z,\qpb}\rightarrow\mathbf{V}_{z,\qpb}$ in the following way: let $X_{z}$ be the base change of $X^{\textup{univ}}$ to $z$, then the rational Dieudonne module $\mathbb{D}(X_z)_{\qpb}\coloneqq\mathbb{D}(X_z)(\zpb)\otimes_{\zpb}\qpb$ has a natural Frobenius-linear isomorphism $\mathbf{F}_z:\mathbb{D}(X_z)_{\qpb}\rightarrow\mathbb{D}(X_z)_{\qpb}$ induced by the relative Frobenius morphism on $X_z$. The morphism $\mathbf{F}_z$ induces a Frobenius-linear isomorphism $\mathbf{F}_z^{\otimes}$ on the tensor algebra $\mathbb{D}(X_z)_{\qpb}^{\otimes}$. In particular, let $\mathbf{F}_z^{(2,2)}$ be its restriction on $\mathbb{D}(X_z)_{\qpb}^{\otimes(2,2)}\simeq\textup{End}_{\qpb}\left(\textup{End}_{\qpb}(\mathbb{D}(X_z)_{\qpb})\right)$. We have
\begin{equation*}
    \mathbf{F}_z^{(2,2)}:\textup{End}_{\qpb}(\mathbb{D}(X_z)_{\qpb})\rightarrow\textup{End}_{\qpb}(\mathbb{D}(X_z)_{\qpb}),\,\,f\mapsto \mathbf{F}_z\circ f\circ\mathbf{F}_z^{-1}.
\end{equation*}
By the definition of $\mathcal{N}^{o}$ in $\S$\ref{gspin-rz}, the crystalline Tate tensor $\bpi_{\textup{crys},z}$ is Frobenius invariant. Hence $\mathbf{F}_z^{(2,2)}$ acts on $\mathbf{V}_{z,\qpb}=\textup{im}\,\bpi_{\textup{crys},z}(\textup{End}_{\qpb}(\mathbb{D}(X_z)_{\qpb}))$ as a Frobenius-linear isomorphism. We use $\Phi_z$ to denote this action.
\par
Let $\Tilde{z}\in\mathcal{N}^{o}(\zpb)$ be a lift of $z$ to $\zpb$. By discussion in $\S$\ref{crystal-over-ortho}, the point $\Tilde{z}$ induces an isotropic line $L^{1}=\varprojlim\limits_{n}\textup{Fil}^{1}\mathbf{V}_{\textup{crys},z}(\zpb/p^{n})\subset\mathbf{V}_{z}$. Let $L^{0}\coloneqq (L^{1})^{\bot}\subset\mathbf{V}_z$, we have a filtration $0\subset L^{1}\subset L^{0}\subset \mathbf{V}_z$, the following lemma proves that this filtration is strongly divisible in the sense of \cite[$\S$3.1]{Laf80} (see also \cite[$\S$4.8]{Mad16}).
\begin{lemma}
    The Frobenius-linear map $\Phi_z$ on $\mathbf{V}_{z,\qpb}$ induces the following isomorphism of $\zpb$-lattices,
    \begin{equation*}
        \Phi_z: p^{-1}L^{1}+L^{0}+p\mathbf{V}_z\stackrel{\sim}\longrightarrow\mathbf{V}_z
    \end{equation*}
    \label{strongdiv}
\end{lemma}
\begin{proof}
The proof is basically the same as that of \cite[$\S$4.8]{Mad16}. Here we give a self-contained argument without referring to results of \cite[$\S$4.2]{Laf80}.
    For simplicity, we use $H$ to denote the Dieudonne module $\mathbb{D}(X_z)(\zpb)\simeq\mathbb{D}(X_{\tilde{z}})(\zpb)$. Let $H_{\mathbb{F}}=H\otimes_{\zpb}\mathbb{F}$. Let $\mathcal{V}_z=p\cdot\mathbf{F}_z^{-1}$ be the Verschiebung morphism on $H$. Then the supersinglarity of $X_z$ implies that $\mathbf{F}_z^{2}(H)=pH$, hence $\mathcal{V}_{z}(H)=\mathbf{F}_z(H)$.
    \par
    Let $0\subset\textup{Fil}^{1}H_{\mathbb{F}}\subset H_{\mathbb{F}}$ be the Hodge filtration given by $X_z$. It is well-known that this filtration is equal to $0\subset\mathcal{V}_z(H)/pH\subset H_{\mathbb{F}}$ (see Oda's work \cite[Corollary 5.11]{Oda69}). By Grothendieck-Messing theory, the Hodge filtration $0\subset\textup{Fil}^{1}H\subset H $ determined by the $p$-divisible group $X_{\Tilde{z}}$ lifts the filtration $0\subset\textup{Fil}^{1}H_{\mathbb{F}}\subset H_{\mathbb{F}}$. Therefore we have $\mathbf{F}_z(H)=\mathcal{V}_z(H)=\textup{Fil}^{1}H+pH$. Combining this with $\mathbf{F}_z^{2}(H)=pH$, we get
    \begin{equation}
        \mathbf{F}_z\left(p^{-1}\textup{Fil}^{1}H+H\right)=H.
        \label{fil1}
    \end{equation}
    \par
    We still use $\mathcal{V}_z$ to denote the base change morphism $\mathcal{V}_z\otimes_{\zpb}\qpb:H_{\qpb}\rightarrow H_{\qpb}$. It's a Frobenius-conjugate linear isomorphism. Then $\mathcal{V}_z^{\vee}:H_{\qpb}^{\vee}\rightarrow H_{\qpb}^{\vee},\,\,f\mapsto f\circ \mathcal{V}_z$ is a Frobenius-linear isomorphism, we define $\mathbf{F}_z^{\vee}$ similarly. The filtration $0\subset\textup{Fil}^{1}H\subset H $ induces a filtration on $H^{\vee}: 0\rightarrow\textup{Fil}^{0}H^{\vee}\rightarrow H^{\vee}$ where $\textup{Fil}^{0}H^{\vee}\simeq(H/\textup{Fil}^{1}H)^{\vee}$. The equality (\ref{fil1}) induces
    \begin{equation}
        \mathcal{V}_z^{\vee}\left(p^{-1}\textup{Fil}^{0}H^{\vee}+H^{\vee}\right)=H^{\vee}.
        \label{fil2}
    \end{equation}
    It's easy to check that the Frobenius-linear action $\mathbf{F}_z^{(2,2)}$ on $\textup{End}_{\qpb}(H_{\qpb})\simeq H_{\qpb}\otimes_{\qpb}H^{\vee}_{\qpb}$ we defined before equals to $p^{-1}\mathbf{F}_z\otimes\mathcal{V}_z^{\vee}$. Therefore (\ref{fil1}) and (\ref{fil2}) together imply that
    \begin{equation*}
        \mathbf{F}_z^{(2,2)}\left(p^{-1}\textup{Fil}^{1}\textup{End}(H)+\textup{Fil}^{0}\textup{End}(H)+p\textup{End}(H)\right)=\textup{End}(H).
    \end{equation*}
    Applying the projector $\bpi_{\textup{crys},z}$, notice that it is Frobenius invariant and respects the Hodge filtration, we get
    \begin{equation*}
        \Phi_z\left(p^{-1}L^{1}+L^{0}+p\mathbf{V}_z\right)=\mathbf{V}_z.
    \end{equation*}
\end{proof}
\begin{lemma}
    Let $\nu_{p}: \qpb\rightarrow\mathbb{Z}\cup\{\infty\}$ be the $p$-adic valuation of $\qpb$. Let $\mathbf{V}$ be a free $\zpb$-module of finite rank equipped with a non-degenerate bilinear pairing $(\cdot,\cdot)$. Let $L^{1}\subset\mathbf{V}$ be an isotropic line, and $L^{0}=(L^{1})^{\perp}\subset\mathbf{V}$. Let $l\in\mathbf{V}$ be a generator of the isotropic line $L^{1}$, then 
    \begin{equation*}
        L^{0}+p\mathbf{V}=\{x\in\mathbf{V}:\nu_{p}((x,l))\geq1\}.
    \end{equation*}
    \label{pairingeq1}
\end{lemma}
\begin{proof}
    By the definition of $L^{0}$, we have $\nu_{p}((x,l))\geq1$ for any $x\in L^{0}+p\mathbf{V}$. Now if $x\in\mathbf{V}$ satisfies the condition that $\nu_{p}((x,l))\geq1$, there exists $x^{\prime}\in\mathbf{V}$ such that $(l,x^{\prime})=p^{-1}(l,x)$ since $l\notin p\mathbf{V}$ and $\mathbf{V}$ is a self-dual quadratic lattice. Therefore $x-p\cdot x^{\prime}\in L^{0}$, hence $x\in L^{0}+p\mathbf{V}$.
\end{proof}

\section{Formally smoothness locus of the special cycle}
\label{smooth}
\subsection{Formally smoothness over a discrete valuation ring}
\begin{lemma}
    Let $n\geq2$ be an integer. Let $R=\mathcal{O}[[t_{1},\cdot\cdot\cdot,t_{n-1}]]$ where $\mathcal{O}$ is a discrete valuation ring of characteristic $(0,p)$ with uniformizer $\pi$ and residue field $\mathbf{k}$. The ring $R$ has maximal ideal $\mathfrak{m}_{R}\coloneqq (\pi,t_{1},\cdot\cdot\cdot,t_{n-1})$. Let $f_{1},f_{2},\cdots,f_{r}$ be $r$ elements in $\mathfrak{m}_{R}$. The quotient ring $\overline{R}=R/(f_{1},\cdots,f_{r})$ is formally smooth over $\mathcal{O}$ of relative dimension $n-r-1$ if and only if its base change $\overline{R}_{\mathbf{k}}=\overline{R}\otimes_{\mathcal{O}}\mathbf{k}$ is formally smooth over $\mathbf{k}$ of relative dimension $n-r-1$.
    \label{smoothlemma}
\end{lemma}
\begin{proof}
    Let $\mathbf{0}: R\rightarrow\mathcal{O}$ be the continuous homomorphism which sends all the $t_{j}$ to $0$. Let $J=(\frac{\partial f_{i}}{\partial t_{j}}(\mathbf{0}))_{\substack{1\leq i\leq r\\1\leq j\leq n-1}}$ be the Jacobian matrix of $f_{1},f_{2},\cdots,f_{r}$ at $\mathbf{0}$. The quotient ring $\overline{R}$ is formally smooth over $\mathcal{O}$ of relative dimension $n-r-1$ if and only if the matrix $J$ has a $r\times r$-minor which is invertible in $\mathcal{O}$, because in this case we can rearrange a system of parameters $t_{1}^{\prime},\cdots,t_{n-1}^{\prime}$ of $R$ such that $t_{i}^{\prime}=f_{i}$ for $1\leq i\leq r$.
    \par
    Let $J_{\mathbf{k}}=J\otimes_{\mathcal{O}}\mathbf{k}$ be the base change of $J$ to $\mathbf{k}$, the matrix $J$ has a $r\times r$-minor which is invertible in $\mathcal{O}$ is equivalent to the rank of $J_{\mathbf{k}}$ is $r$, which is further equivalent to $\overline{R}_{\mathbf{k}}$ is formally smooth over $\mathbf{k}$ of relative dimension $n-r-1$, hence the quotient ring $\overline{R}=R/(f_{1},\cdots,f_{r})$ is formally smooth over $\mathcal{O}$ of relative dimension $n-r-1$ if and only if $\overline{R}_{\mathbf{k}}$ is formally smooth over $\mathbf{k}$ of relative dimension $n-r-1$.
\end{proof}
\subsection{The unitary case}
For a point $z\in\mathcal{N}^{u}(\mathbb{F})$, let $\widehat{\mathcal{N}}^{u}_{z}$ be the completion of the formal scheme $\mathcal{N}^{u}$ at $z$. Recall that the Dieudonne module $\mathbb{D}(X_z)$ of $X_z$ has a decomposition $\mathbb{D}(X_z)=\mathbf{D}_z\oplus\overline{\mathbf{D}}_z$ as in $\S$\ref{def-unitary}. Both $\mathbf{D}_z$ and $\overline{\mathbf{D}}_z$ are free $\ofb$-modules of rank $n$. Let $L\subset\mathbb{V}^{u}$ be an $\mathcal{O}_{E}$-lattice. For a point $z\in\mathcal{Z}(L)(\mathbb{F})$, we defined in (\ref{crystal-unitary}) an $\mathcal{O}_F$-linear map $i_{\textup{crys},z}:L\rightarrow\mathbf{D}_z$. We say the map $i_{\textup{crys},z}$ is primitive at $z$ if the induced map $\overline{i_{\textup{crys},z}}: L/\varpi L\rightarrow \mathbf{D}_{z}/\varpi \mathbf{D}_{z}$ is injective.
\begin{lemma}
    Let $L\subset\mathbb{V}^{u}$ be an $\mathcal{O}_{E}$-lattice of rank $1\leq r\leq n-1$ and $z\in\mathcal{Z}(L)(\mathbb{F})$, then $\mathcal{Z}(L)$ is formally smooth over $\ofb$ of relative dimension $n-r-1$ at $z$ if and only if the map $i_{\textup{crys},z}$ is primitive at $z$.
    \label{smoothcycleu}
\end{lemma}
\begin{proof}
    The special cycle $\mathcal{Z}(L)$ is cut out by $r$ equations in the deformation space $\widehat{\mathcal{N}}^{u}_{z}\simeq\textup{Spf}\,\mathcal{O}_{\mathcal{N}^{u},z}$ by Proposition \ref{divisor}. Therefore the special cycle $\mathcal{Z}(L)$ is formally smooth over $\ofb$ of relative dimension $n-r-1$ at $z$ if and only if $\mathcal{Z}(L)_{\mathbb{F}}\coloneqq\mathcal{Z}(L)\times_{\ofb}\mathbb{F}$ is formally smooth over $\mathbb{F}$ of relative dimension $n-r-1$ at $z$ by Lemma \ref{smoothlemma}. Let $\mathcal{N}^{u}_{\mathbb{F}}=\mathcal{N}^{u}\times_{\ofb}\mathbb{F}$. Let $l\in\overline{\mathbf{D}}_z$ be an element such that its image $\overline{l}$ in $\overline{\mathbf{D}}_{z}/\varpi\overline{\mathbf{D}}_{z}\simeq\overline{\mathbf{D}}_{z}(\mathbb{F})$ is a generator of the line $\textup{Fil}^{1}\mathbb{D}(X_z)(\mathbb{F})\cap\overline{\mathbf{D}}_{z}(\mathbb{F})$. Let $\mathbb{F}[\epsilon]=\mathbb{F}[X]/X^{2}$. 
    \par
    By Schlessinger's criterion in \cite[Theorem 2.11]{Sch68}, the tangent space $\textup{Tgt}_{z}(\mathcal{N}^{u}_{\mathbb{F}})$ of $\mathcal{N}^{u}_{\mathbb{F}}$ at $z$ can be identified with the set $\widehat{\mathcal{N}}^{u}_{z}(\mathbb{F}[\epsilon])$. This set has a natural bijection to the set of lines in $\overline{\mathbf{D}}_z(\mathbb{F})\otimes_{\mathbb{F}}\mathbb{F}[\epsilon]$ which lifts the line $\mathbb{F}\cdot\overline{l}$ in $\overline{\mathbf{D}}_z(\mathbb{F})$ by Lemma \ref{deformu}. Therefore any such line has a generator of the form $\overline{l}+\epsilon\cdot w$ for some element $w\in\overline{\mathbf{D}}_z(\mathbb{F})$. Two elements $\overline{l}+\epsilon\cdot w$ and $\overline{l}+\epsilon\cdot w^{\prime}$ generate the same line if and only if $w-w^{\prime}\in\mathbb{F}\cdot\overline{l}$, hence $\textup{Tgt}_{z}(\mathcal{N}^{u}_{\mathbb{F}})\simeq\overline{\mathbf{D}}_z(\mathbb{F})/\mathbb{F}\cdot\overline{l}$, and $\textup{dim}_{\mathbb{F}}\textup{Tgt}_{z}(\mathcal{N}^{u}_{\mathbb{F}})=n-1$.
    \par
    The tangent space $\textup{Tgt}_{z}(\mathcal{Z}(L)_{\mathbb{F}})$ of $\mathcal{Z}(L)_{\mathbb{F}}$ at $z$ can be identified with the set $\widehat{\mathcal{Z}(L)}_{z}(\mathbb{F}[\epsilon])$. For all $x\in L$, let $\overline{x_{\textup{crys},z}}$ be the image of $x_{\textup{crys},z}$ in $\mathbf{D}_z(\mathbb{F})\simeq\mathbf{D}_z/\varpi\mathbf{D}_z$. The set $\widehat{\mathcal{Z}(L)}_{z}(\mathbb{F}[\epsilon])$ has a natural bijection to the set of lines in $\overline{\mathbf{D}}_z(\mathbb{F})\otimes_{W}\mathbb{F}[\epsilon]$ which lifts the line $\mathbb{F}\cdot\overline{l}$ in $\overline{\mathbf{D}}_z(\mathbb{F})$ and orthogonal to $\overline{x_{\textup{crys},z}}$ for all $x\in L$. Let $\overline{l}+\epsilon\cdot w$ be a generator of a line in the set $\widehat{\mathcal{Z}(L)}_{z}(\mathbb{F}[\epsilon])$, then $\langle\overline{x_{\textup{crys},z}},\overline{l}+\epsilon\cdot w\rangle=\epsilon\langle\overline{x_{\textup{crys},z}},w\rangle=0$, hence $\langle\overline{x_{\textup{crys},z}},w\rangle=0$. Therefore $\textup{Tgt}_{z}(\mathcal{Z}(L)_{\mathbb{F}})$ is the subspace of $\overline{\mathbf{D}}_{\mathbb{F}}/\mathbb{F}\cdot\overline{l}$ orthogonal to the image of $L$ in $\mathbf{D}_z(\mathbb{F})$ under the perfect pairing $\langle\cdot,\cdot\rangle:\mathbf{D}_z(\mathbb{F})\times\overline{\mathbf{D}}_z(\mathbb{F})\rightarrow\mathbb{F}$. Let $\overline{L}\subset\mathbf{D}_z/\varpi\mathbf{D}_z$ be the $\mathbb{F}$-linear subspace spanned by the image of $L$ in $\mathbf{D}_{z}(\mathbb{F})$, we have
    \begin{equation}
        \textup{dim}_{\mathbb{F}}\textup{Tgt}_{z}(\mathcal{Z}(L)_{\mathbb{F}})= n-1-\textup{dim}_{\mathbb{F}}(\overline{L}).
        \label{dimtangent}
    \end{equation}
    therefore $\mathcal{Z}(L)$ is formally smooth over $\ofb$ of relative dimension $n-r-1$ at $z$ if and only if $\textup{dim}_{\mathbb{F}}(\overline{L})=r$, which is equivalent to the fact that the map $\overline{i_{\textup{crys},z}}: L/\varpi L\rightarrow \mathbf{D}_z/\varpi\mathbf{D}_z$ is injective, i.e., the map $i_{\textup{crys},z}$ is primitive.
\end{proof}

\subsection{The GSpin case}
For a point $z\in\mathcal{N}^{o}(\mathbb{F})$, let $\widehat{\mathcal{N}}^{o}_{z}$ be the completion of the formal scheme $\mathcal{N}^{o}$ at $z$. Let $\mathbf{V}_z=\mathbf{V}_{\textup{crys},z}$ be the free $\zpb$-module of rank $m$ defined in $\S$\ref{crystal-over-ortho}. Let $L\subset\mathbb{V}^{o}$ be an $\mathbb{Z}_{p}$-lattice. For a point $z\in\mathcal{Z}(L)(\mathbb{F})$, we defined in (\ref{crystalline-gspin}) a $\zp$-linear map $i_{\textup{crys},z}:L\rightarrow\mathbf{V}_z$. We say the map $i_{\textup{crys},z}$ is primitive at $z$ if the induced map $\overline{i_{\textup{crys},z}}: L/pL\rightarrow \mathbf{V}_z/p\mathbf{V}_z$ is injective.
\begin{lemma}
    Let $L\subset\mathbb{V}^{o}$ be a $\zp$-lattice of rank $r\geq1$ and $z\in\mathcal{Z}(L)(\mathbb{F})$, then $\mathcal{Z}(L)$ is formally smooth over $\zpb$ of relative dimension $n-r-1$ at $z$ if and only if the following two assertions hold,
    \begin{itemize}
        \item [(i)]The map $i_{\textup{crys},z}$ is primitive at $z$;
        \item[(ii)] There exists a lift of $z$ to $z^{\prime}\in\mathcal{Z}(L)(\zpb/p^{2})$.
    \end{itemize}
    \label{smoothcycleo}
\end{lemma}
\begin{proof}
    The special cycle $\mathcal{Z}(L)$ is cut out by $r$ equations in the deformation space $\widehat{\mathcal{N}}^{o}_{z}\simeq\textup{Spf}\,\mathcal{O}_{\mathcal{N},z}$ by Proposition \ref{divisor}. Therefore the special cycle $\mathcal{Z}(L)$ is formally smooth over $\zpb$ of relative dimension $n-r-1$ at $z$ if and only if $\mathcal{Z}(L)_{\mathbb{F}}\coloneqq\mathcal{Z}(L)\times_{\zpb}\mathbb{F}$ is formally smooth over $\mathbb{F}$ of relative dimension $n-r-1$ at $z$ by Lemma \ref{smoothlemma}. Let $\mathcal{N}^{o}_{\mathbb{F}}=\mathcal{N}^{o}\times_{\zpb}\mathbb{F}$. Let $l\in\mathbf{V}_z$ be an isotropic element such that its image $\overline{l}\in\mathbf{V}_z/p\mathbf{V}_z\simeq\vcrysz(\mathbb{F})$ generates the line $\textup{Fil}^{1}\vcrysz(\mathbb{F})$. Let $(\cdot,\cdot)$ be the bilinear form on $\vcrysz(\mathbb{F})$. Let $\mathbb{F}[\epsilon]=\mathbb{F}[X]/X^{2}$.  
    \par
    By Schlessinger's criterion in \cite[Theorem 2.11]{Sch68}, the tangent space $\textup{Tgt}_{z}(\mathcal{N}^{o}_{\mathbb{F}})$ (resp. $\textup{Tgt}_{z}(\mathcal{Z}(L)_{\mathbb{F}})$) of $\mathcal{N}^{o}_{\mathbb{F}}$ (resp. $\mathcal{Z}(L)_{\mathbb{F}}$) at $z$ can be identified with the set $\widehat{\mathcal{N}}_{z}^{o}(\mathbb{F}[\epsilon])$. This set has a natural bijection to the set of isotropic lines in $\vcrys(\mathbb{F}[\epsilon])\simeq\vcrysz(\mathbb{F})\otimes_{\mathbb{F}}\mathbb{F}[\epsilon]$ which lifts the line $\mathbb{F}\cdot \overline{l}$ in $\vcrysz(\mathbb{F})$ by Lemma \ref{deformo}. Therefore any such line has a generator of the form $\overline{l}+\epsilon\cdot w$ for some $w\in\vcrysz(\mathbb{F})$, note that
    \begin{equation*}
        0=(\overline{l}+\epsilon\cdot w,\overline{l}+\epsilon\cdot w)=(\overline{l},\overline{l})+2\epsilon(\overline{l},w)=2\epsilon(\overline{l},w),
    \end{equation*}
    hence the fact that the line generated by $\overline{l}+\epsilon\cdot w$ is isotropic is equivalent to the vector $w\in\vcrysz(\mathbb{F})$ is orthogonal to $\overline{l}$. Two elements $\overline{l}+\epsilon\cdot w$ and $\overline{l}+\epsilon\cdot w^{\prime}$ generate the same line if and only if $w^{\prime}-w\in\mathbb{F}\cdot\overline{l}$, hence $ \textup{Tgt}_{z}(\mathcal{N}^{o}_{\mathbb{F}})\simeq\{\overline{l}\}^{\bot}/\mathbb{F}\cdot\overline{l}$, and $\textup{dim}_{\mathbb{F}}\textup{Tgt}_{z}(\mathcal{N}^{o}_{\mathbb{F}})=n-1$. There is a bilinear pairing on the space $\{\overline{l}\}^{\bot}/\mathbb{F}\cdot\overline{l}$ induced by $(\cdot,\cdot)$, we denote it by $\overline{(\cdot,\cdot)}$. It is also non-degenerate since because the pairing $(\cdot,\cdot)$ is non-degenerate.
\par
The tangent space $\textup{Tgt}_{z}(\mathcal{Z}(L)_{\mathbb{F}})$ of $\mathcal{Z}(L)_{\mathbb{F}}$ at $z$ can be identified with the set $\widehat{\mathcal{Z}(L)}_{z}(\mathbb{F}[\epsilon])$. For all $x\in L$, let $\overline{x_{\textup{crys},z}}$ be the image of $x_{\textup{crys},z}$ in $\vcrysz(\mathbb{F})$. The set $\widehat{\mathcal{Z}(L)}_{z}(\mathbb{F}[\epsilon])$ has a natural bijection to the set of isotropic lines in $\vcrysz(\mathbb{F})\otimes_{\mathbb{F}}\mathbb{F}[\epsilon]$ which lifts the line $\mathbb{F}\cdot\overline{l}$ in $\vcrysz(\mathbb{F})$ and orthogonal to $\overline{x_{\textup{crys},z}}$ for any $x\in L$. Let $\overline{l}+\epsilon\cdot w$ be a generator of a line in the set $\widehat{\mathcal{Z}(L)}_{z}(\mathbb{F}[\epsilon])$, then $(\overline{x_{\textup{crys},z}},\overline{l}+\epsilon\cdot w)=\epsilon(\overline{x_{\textup{crys},z}},w)=0$, hence $(\overline{x_{\textup{crys},z}},w)=0$. Therefore $\textup{Tgt}_{z}(\mathcal{Z}(L)_{\mathbb{F}})$ is the subspace of $\{\overline{l}\}^{\bot}/\mathbb{F}\cdot\overline{l}$ orthogonal to the image of $L$ in $\vcrysz(\mathbb{F})$ under the pairing $\overline{(\cdot,\cdot)}$. Let $\overline{L}\subset\mathbf{V}_z/p\mathbf{V}_z$ be the $\mathbb{F}$-linear subspace spanned by the image of $L$ in $\vcrysz(\mathbb{F})$, it is contained in the subspace $\{\overline{l}\}^{\bot}$ of $\vcrysz(\mathbb{F})$, we have
\begin{equation*}
    \textup{Tgt}_{z}(\mathcal{Z}(L)_{\mathbb{F}})=\left((\overline{L}+\mathbb{F}\cdot\overline{l})/\mathbb{F}\cdot\overline{l}\right)^{\bot}\subset\{\overline{l}\}^{\bot}/\mathbb{F}\cdot\overline{l}.
\end{equation*}
Therefore $\textup{dim}_{\mathbb{F}}\,\textup{Tgt}_{z}(\mathcal{Z}(L)_{\mathbb{F}})=n-1-\textup{dim}_{\mathbb{F}}\,\left((\overline{L}+\mathbb{F}\cdot\overline{l})/\mathbb{F}\cdot\overline{l}\right)$ since the pairing $\overline{(\cdot,\cdot)_{\mathbb{F}}}$ is non-degenerate on $\{\overline{l}\}^{\bot}/\mathbb{F}\cdot\overline{l}$. Therefore $\mathcal{Z}(L)$ is formally smooth over $\zpb$ of relative dimension $n-r-1$ at $z$ if and only if $\textup{dim}_{\mathbb{F}}\,\left((\overline{L}+\mathbb{F}\cdot\overline{l})/\mathbb{F}\cdot\overline{l}\right)=r$, which is further equivalent to $\textup{dim}_{\mathbb{F}}\,\overline{L}=r$ and $\overline{l}\notin\overline{L}$. The condition $\textup{dim}_{\mathbb{F}}\,\overline{L}=r$ is equivalent to (i) which says that the map $i_{\textup{crys},z}$ is primitive at $z$.
\par
For one direction, if the formal scheme $\mathcal{Z}(L)$ is formally smooth over $\zpb$ of relative dimension $n-r-1$ at $z$, clearly $i_{\textup{crys},z}$ is primitive by the above discussion. The formally smoothness of $\mathcal{Z}(L)$ over $\zpb$ also implies that an $\mathbb{F}$-point $z\in\mathcal{Z}(\mathbb{F})$ can be lifted to a point $z^{\prime}\in\mathcal{Z}(L)(\zpb/p^{2})$. Hence (i) and (ii) are true.
\par
    Let's consider the other direction. Assume that (i) and (ii) are true, we want to show that the element $\overline{l}$ doesn't belong to the space $\overline{L}$. Let $z^{\prime}\in\mathcal{Z}(L)(\zpb/p^{2})$ be a lift of $z$ to $\zpb/\varpi^{2}$. By the formally smoothness of $\mathcal{N}^{o}$ at $z$, there exists a point $\Tilde{z}\in\widehat{\mathcal{N}}^{o}_{z}(\zpb)$ lifting the point $z^{\prime}\in\widehat{\mathcal{N}}^{o}_{z}(\zpb/p^{2})$. The point $\tilde{z}$ corresponds to an isotropic line $L^{1}$. Let $l\in \mathbf{V}_z$ be a generator of $L^{1}$ whose image in $\vcrysz(\mathbb{F})$ is $\overline{l}$. Let $\{x_{i}\}_{i=1}^{r}$ be an $\zp$-basis of $L$. If $\overline{l}$ belongs to the space $\overline{L}$, there exist $a_i\in\zpb$ for $1\leq i\leq r$ and $v\in\mathbf{V}_z$ such that
    \begin{equation}
        l=\sum\limits_{i=1}^{r}a_ix_{i,\textup{crys},z}+p\cdot v.
        \label{l}
    \end{equation}
    Notice that $\overline{l}=\sum\limits_{i=1}^{r}\overline{a_{i}}\cdot\overline{x_{i,\textup{crys},z}}\neq0$, hence there exists at least one $i$ such that $a_{i}\in \zpb^{\times}$ and $x_{i,\textup{crys},z}\notin p\mathbf{V}_z$. Let $\Phi_z$ be the Frobenius action on $\mathbf{V}_z$. If $v\in L^{0}+p\mathbf{V}_z$, then $\Phi_z(v)\in \mathbf{V}_z$ by Lemma \ref{strongdiv} The same lemma also implies that $\Phi_z(l)\in p\mathbf{V}_z$. Therefore $\Phi_z(\sum\limits_{i=1}^{r}a_ix_{i,\textup{crys},z})=\Phi_z(l)-p\Phi_z(v)\in p\mathbf{V}_z$. However,
    \begin{equation*}
        \Phi_z(\sum\limits_{i=1}^{r}a_ix_{i,\textup{crys},z})=\sum\limits_{i=1}^{r}\sigma(a_{i})\cdot\Phi_z(x_{i,\textup{crys},z})=\sum\limits_{i=1}^{r}\sigma(a_{i})\cdot x_{i,\textup{crys},z},
    \end{equation*}
    hence $\Phi_z(\sum\limits_{i=1}^{r}a_ix_{i,\textup{crys},z})\notin p\mathbf{V}_z$ because there exists an integer $1\leq i\leq r$ such that $\sigma(a_{i})\in \zpb^{\times}$ and $x_{i,\textup{crys},z}\notin p\mathbf{V}_z$, hence $v\notin L^{0}+p\mathbf{V}_z$. Therefore $(l,v)\in \zpb^{\times}$ by Lemma \ref{pairingeq1}.
    \par
   Recall that $\nu_{p}$ is the $p$-adic valuation on $\qpb$. We claim that $\nu_{p}((l,x_{i,\textup{crys},z}))\geq2$ for all $1\leq i\leq r$. Since the point $\Tilde{z}\in\mathcal{N}^{o}(\zpb)$ is a lift of $z^{\prime}$, the image of $l$ in $\vcrysz(\zpb/p^{2})\simeq\mathbf{V}_z/p^{2}\mathbf{V}_z$ generates the line $\textup{Fil}^{1}\vcrysz(\zpb/p^{2})$ corresponding to $z^{\prime}$. By definition, the image of $x_{i,\textup{crys},z}\in\mathbf{V}_z$ in $\vcrysz(\zpb/p^{2})$ is $x_{i,\textup{crys},z}(\zpb/p^{2})$. By Lemma \ref{deformo}, the line $\textup{Fil}^{1}\vcrysz(\zpb/p^{2})$ is orthogonal to $x_{\textup{crys},z}(\zpb/p^{2})$, which is equivalent to $(l,x_{i,\textup{crys},z})=0\in\zpb/(p^{2})$, i.e., $\nu_{p}((l,x_{i,\textup{crys},z}))\geq2$ for all $1\leq i\leq r$.
   Therefore $\nu_{p}(\sum\limits_{i=1}^{r}a_{i}(l,x_{i,\textup{crys},z}))\geq2$. However, the pairing $(l,\sum\limits_{i=1}^{r}a_ix_{i,\textup{crys},z})=(l,l-p\cdot v)=-p(l,v)$ by (\ref{l}), hence $\nu_{p}((l,\sum\limits_{i=1}^{r}a_ix_{i,\textup{crys},z}))=1$ since $(l,v)\in \zpb^{\times}$ by previous discussion, this is a contradiction. Therefore the element $\overline{l}$ doesn't belong to the space $\overline{L}$ if (i) and (ii) are true, hence the formal scheme $\mathcal{Z}(L)$ is formally smooth over $\zpb$ of relative dimension $n-r-1$ at $z$.
\end{proof}
\begin{remark}
    In the above proof, we only use the inclusion part of the strong divisibility Lemma \ref{strongdiv}, i.e., $\Phi_z\left(p^{-1}L^{1}+L^{0}+p\mathbf{V}_z\right)\subset\mathbf{V}_z.$
\end{remark}

\section{Regularity of the difference divisor}
\label{regularity-part}
\subsection{Commutative algebra preparation}
\begin{lemma}
    Let $n\geq2$ be an integer. Let $R=\mathcal{O}[[t_{1},\cdot\cdot\cdot,t_{n-1}]]$ where $\mathcal{O}$ is a discrete valuation ring of characteristic $(0,p)$ with uniformizer $\pi$ and $\pi$-adic valuation $\nu_{\pi}$. The ring $R$ has maximal ideal $\mathfrak{m}_{R}\coloneqq (\pi,t_{1},\cdot\cdot\cdot,t_{n-1})$. In the following the symbol $(\textup{unit})$ means an element in $R^{\times}$.
    \begin{itemize}
        \item[$\bullet$] Let $g\in \mathfrak{m}_{R}$. If for any ring homomorphism $f:R\rightarrow\mathcal{O}$, we have $\nu_{\pi}(f(g))=1$, then $g\equiv (\textup{unit})\cdot\pi \,\,(\textup{mod}\,\,\mathfrak{m}_{R}^{2})$.
        \item[$\bullet$] Let $d\in\mathfrak{m}_{R}\backslash\mathfrak{m}_{R}^{2}$. Let $h$ be an element in $\mathfrak{m}_{R}$ such that $h\nequiv (\textup{unit})\cdot\pi \,\,(\textup{mod}\,\,\mathfrak{m}_{R}^{2})$. Then there exists a ring homomorphism $f:R\rightarrow\mathcal{O}$ such that $f(d)\neq0$ and $\nu_{\pi}(f(h))\geq2$.
    \end{itemize}
    \label{tech1}
\end{lemma}
\begin{proof}
    Notice that any ring homomorphism $f:R\rightarrow\mathcal{O}$ maps $t_i$ into the maximal ideal $(\pi)\subset\mathcal{O}$.
    For the first assertion, let $g\equiv a_{0}\pi+\sum\limits_{j=1}^{n-1}a_{j}t_{j}\,\,(\textup{mod}\,\,\mathfrak{m}_{R}^{2})$ where $a_{j}\in\mathcal{O}$ for every $0\leq j \leq n-1$. If $\nu_{\pi}(a_{0})\geq1$, consider the continuous homomorphism $f:R\rightarrow\mathcal{O}$ such that $f(t_{j})=\pi^{2}$ for every $0\leq j \leq n-1$, then $\nu_{\pi}(f(g))\geq2$, which is a contradiction. Therefore $\nu_{\pi}(a_{0})=0$. If $g\nequiv (\textup{unit})\cdot\pi \,\,(\textup{mod}\,\,\mathfrak{m}_{R}^{2})$, then there exists at least one $j\geq1$ such that $\nu_{\pi}(a_{j})=0$. Suppose that $\nu_{\pi}(a_{j_{0}})=0$ for some $j_{0}\geq1$, consider the continuous homomorphism $f:R\rightarrow\mathcal{O}$ such that $f(t_{j_{0}})=-a_{j_{0}}^{-1}a_{0}\pi$ and $f(t_{j})=\pi^{2}$ for $j \neq j_{0}$, then $\nu_{\pi}(f(g))\geq2$, which is a contradiction, hence $g\equiv (\textup{unit})\cdot\pi \,\,(\textup{mod}\,\,\mathfrak{m}_{R}^{2})$.
    \par
    For the second assertion, we consider the following two cases.
    \begin{itemize}
        \item[Case 1.] $d\equiv (\textup{unit})\cdot\pi \,\,(\textup{mod}\,\,\mathfrak{m}_{R}^{2})$. Since $h\nequiv (\textup{unit})\cdot\pi \,\,(\textup{mod}\,\,\mathfrak{m}_{R}^{2})$, there exists a continuous ring homomorphism $f:R\rightarrow\mathcal{O}$ such that $\nu_{\pi}(f(h))\geq2$ by the proof of the first assertion, for this $f$, we have $\nu_{\pi}(f(d))=1$, hence $f(d)\neq0$.
        \item[Case 2.] $d\nequiv (\textup{unit})\cdot\pi \,\,(\textup{mod}\,\,\mathfrak{m}_{R}^{2})$, then we can choose another system of uniformizers $t_{1},\cdots,t_{n-1}$ such that $d=t_{1}$. Let $h\equiv b_{0}\pi+\sum\limits_{j=1}^{n-1}b_{j}t_{j}\,\,(\textup{mod}\,\,\mathfrak{m}_{R}^{2})$ where $b_{j}\in\mathcal{O}$ for every $0\leq j \leq n-1$. If $\nu_{\pi}(b_{0})\geq1$, consider the continuous homomorphism $f:R\rightarrow\mathcal{O}$ such that $f(t_{j})=\pi^{2}$ for every $0\leq j \leq n-1$, then $\nu_{\pi}(f(h))\geq2$ and $f(d)=f(t_{1})\neq0$. If $\nu_{\pi}(b_{0})=0$, then there exists at least one $j\geq1$ such that $\nu_{\pi}(b_{j})=0$ since $h\nequiv (\textup{unit})\cdot\pi \,\,(\textup{mod}\,\,\mathfrak{m}_{R}^{2})$. Suppose that $\nu_{\pi}(b_{j_{0}})=0$ for some $j_{0}\geq1$, consider the continuous homomorphism $f:R\rightarrow\mathcal{O}$ such that $f(t_{j_{0}})=-b_{j_{0}}^{-1}b_{0}\pi$ and $f(t_{j})=\pi^{2}$ for $j \neq j_{0}$, then $\nu_{\pi}(f(g))\geq2$ and $f(d)=f(t_{1})\neq0$.
    \end{itemize}
\end{proof}
\begin{lemma}
   Let $n\geq2$ be an integer. Let $R=\mathcal{O}[[t_{1},\cdot\cdot\cdot,t_{n-1}]]$ where $\mathcal{O}$ is a discrete valuation ring of characteristic $(0,p)$ with uniformizer $\pi$ and $\pi$-adic valuation $\nu_{\pi}$. Let $(d_{a})_{a\geq0}$ be a sequence of elements in $R$ such that $d_{a}\in\mathfrak{m}_{R}\coloneqq (\pi,t_{1},\cdot\cdot\cdot,t_{n-1})$ for any $a\geq0$ and $d_{0}\in\mathfrak{m}_{R}\backslash\mathfrak{m}_{R}^{2}$. Let $f_{a}=\prod\limits_{i=0}^{a}d_{i}$ and $\mathcal{Z}(f_{a})\coloneqq\textup{Spf}\,R/f_{a}$ be the closed formal subscheme of $\textup{Spf}\,R$. For any morphism $z:\textup{Spf}\,\mathcal{O}\rightarrow\textup{Spf}\,R$, we use $z^{\sharp}$ to denote the corresponding ring homomorphism $R\rightarrow \mathcal{O}$. If for all $a\geq1$ and all morphisms $z:\textup{Spf}\,\mathcal{O}\rightarrow\textup{Spf}\,R$, there exists a Cartesian diagram
   \begin{equation*}
     \xymatrix{
    \textup{Spec}\,\mathcal{O}/(\pi^{a}\cdot z^{\sharp}(d_{0}))\ar[r]\ar[d]&\textup{Spf}\,\mathcal{O}\ar[d]^{z}\\
    \mathcal{Z}(f_{a})\ar[r]&\textup{Spf}\,R.}
\end{equation*}
then $d_{a}\equiv (\textup{unit})\cdot\pi \,\,(\textup{mod}\,\,\mathfrak{m}_{R}^{2})$ for all $a\geq 1$.
\label{technical}
\end{lemma}
\begin{proof}
     Let's assume the contrary that there exists $a\geq1$ such that $d_{a}\nequiv (\textup{unit})\cdot\pi \,\,(\textup{mod}\,\,\mathfrak{m}_{R}^{2})$. Let $k\geq1$ be the least integer such that $d_{i}\equiv (\textup{unit})\cdot\pi \,\,(\textup{mod}\,\,\mathfrak{m}_{R}^{2})$ for all $1\leq i< k$ and $d_{k}\nequiv(\textup{unit})\cdot\pi \,\,(\textup{mod}\,\,\mathfrak{m}_{R}^{2})$, then by Lemma \ref{tech1} there exists a morphism $z:\textup{Spf}\,\mathcal{O}\rightarrow\textup{Spf}\,R$ such that $z^{\sharp}(d_{0})\neq0$ and $\nu_{\pi}(z^{\sharp}(d_{k}))\geq2$, then $\mathcal{Z}(f_{k})\times_{\textup{Spf}\,R,z}\textup{Spf}\,\mathcal{O}$ is cut out in $\textup{Spf}\,\mathcal{O}$ by $(\pi^{m})$ for some integer $m$ such that $ k+\nu_{\pi}(z^{\sharp}(d_{0}))+1\leq m<\infty$, hence the fiber product $\mathcal{Z}(f_{k})\times_{\textup{Spf}\,R,z}\textup{Spf}\,\mathcal{O}$ is not isomorphic to $\textup{Spec}\,\mathcal{O}/(\pi^{k}\cdot z^{\sharp}(d_{0}))$, which is a contradiction. Therefore such $k$ doesn't exist, hence $d_{a}\equiv (\textup{unit})\cdot\pi \,\,(\textup{mod}\,\,\mathfrak{m}_{R}^{2})$ for all $a\geq 1$.
\end{proof}

\subsection{Local equations of special divisors}
In the following, we always fix a geometric point $z\in\mathcal{N}^{u}(\mathbb{F})$ (resp. $\mathcal{N}^{o}(\mathbb{F})$). Denote by $\widehat{\mathcal{N}}^{u}_z$ (resp. $\widehat{\mathcal{N}}^{o}_z$) the formal completion of the formal scheme $\mathcal{N}^{u}$ (resp. $\mathcal{N}^{o}$) at the point $z$. Denote by $\widehat{\mathcal{O}}_{\mathcal{N}^{u},z}$ (resp. $\widehat{\mathcal{O}}_{\mathcal{N}^{o},z}$) the complete local ring of the formal scheme $\mathcal{N}^{u}$ (resp. $\mathcal{N}^{o}$) at this point. Denote by $\mathbf{D}$ and $\overline{\mathbf{D}}$ (resp. $\mathbf{V}$) the $\ofb$-module (resp. $\zpb$-module) $\mathbf{D}_z$ and $\overline{\mathbf{D}}_z$ (resp. $\mathbf{V}_z$) defined in $\S$\ref{def-unitary} (resp. $\S$\ref{crystal-over-ortho}).
\begin{proposition}
    Let $x\in\mathbb{V}^{u}$ (resp. $x\in\mathbb{V}^{o}$) be a nonzero special quasi-homomorphism such that $(x,x)\in\mathcal{O}_{E}$ (resp. $q_{\mathbb{V}^{o}}(x)\in \zp$). Let $z\in\mathcal{N}^{u}(\mathbb{F})$ (resp. $z\in\mathcal{N}^{o}(\mathbb{F})$) be a point. Let $\tilde{z}\in\mathcal{N}^{u}(\ofb)$ (resp. $\mathcal{N}^{o}(\zpb)$) be a lift of $z$ to $\ofb$ and $\Tilde{z}^{\sharp}:\widehat{\mathcal{O}}_{\mathcal{N}^{u},z}\rightarrow \ofb$ (resp. $\Tilde{z}^{\sharp}:\widehat{\mathcal{O}}_{\mathcal{N}^{o},z}\rightarrow \zpb$) be the corresponding ring homomorphism. Let $f_{x,z}\in\widehat{\mathcal{O}}_{\mathcal{N}^{u},z}$ (resp. $\widehat{\mathcal{O}}_{\mathcal{N}^{o},z}$) be the equation of the Cartier divisor $\mathcal{Z}(x)$ at $z$. Assume that $x_{\textup{crys},z}\in\mathbf{V}$ (resp. $x_{\textup{crys},z}\in\mathbf{D}$).
    \begin{itemize}
        \item[$\bullet$] In the unitary case, let $\langle\cdot,\cdot\rangle:\mathbf{D}\times\overline{\mathbf{D}}\rightarrow \ofb$ be the perfect pairing. The lift $\Tilde{z}$ corresponds to a line $L^{1}\subset\overline{\mathbf{D}}$. Let $l$ be a generator of $L^{1}$, then 
    \begin{equation}
    (\Tilde{z}^{\sharp}(f_{x,z}))=(\langle x_{\textup{crys},z},l\rangle)\subset \ofb.
    \label{order}
    \end{equation}
    \item[$\bullet$] In the GSpin case, let $(\cdot,\cdot):\mathbf{V}\times\mathbf{V}\rightarrow \zpb$ be the non-degenerate bilinear pairing on $\mathbf{V}$. The lift $\Tilde{z}$ corresponds to an isotropic line $L^{1}\subset\mathbf{V}$. Let $l$ be a generator of $L^{1}$, then 
    \begin{equation}
    (\Tilde{z}^{\sharp}(f_{x,z}))=\left((x_{\textup{crys},z},l)\right)\subset \zpb.
    \label{ordero}
\end{equation}
    \end{itemize}
\label{orderlemma}
\end{proposition}
\begin{proof}
    We first consider the unitary case. Let $\widehat{\mathcal{Z}(x)}_{z}$ be the completion of the special cycle $\mathcal{Z}(x)$ at $z$, there exists a Cartesian diagram,
    \begin{equation*}
        \xymatrix{
    \textup{Spf}\,\ofb/(\Tilde{z}^{\sharp}(f_{x,z}))\ar[r]\ar[d]&\textup{Spf}\,\ofb\ar[d]^{\Tilde{z}}\\
    \widehat{\mathcal{Z}(x)}_{z}\ar[r]&\widehat{\mathcal{N}}^{u}_{z}.}
    \end{equation*}
    Let $m=\nu_{\varpi}(\Tilde{z}^{\sharp}(f_{x,z}))\in\mathbb{Z}_{\geq0}\cup\{\infty\}$, $t=\nu_{\varpi}(\langle x_{\textup{crys},z},l\rangle)\in\mathbb{Z}_{\geq0}\cup\{\infty\}$, the equality in (\ref{order}) is equivalent to $t= m$.
    \par
    If $m=0$, we have $\textup{Spf}\,\ofb/(\Tilde{z}^{\sharp}(f_{x,z}))=\varnothing$, then $f_{x,z}$ is a unit in the ring $\widehat{\mathcal{O}}_{\mathcal{N}^{u},z}$ $z\notin\mathcal{Z}(x)(\mathbb{F})$ because otherwise $\Tilde{z}^{\sharp}(f_{x,z})\in(\varpi)\subset\ofb$. Hence $z\notin\mathcal{Z}(x)(\mathbb{F})$. We want to show that $t=0$. Let's assume the converse, i.e., $t\geq1$. Let $\mathbb{D}(x):\mathbb{D}(\overline{\mathbb{Y}})\otimes_{\ofb}\Breve{F}\rightarrow\mathbb{D}(X_z)\otimes_{\ofb}\Breve{F}$ be the map on the Dieudonne modules induced by the quasi-homomorphism $x$. Denote by $\Phi$ the Frobenius morphism on $\mathbb{D}(\overline{\mathbb{Y}})$ and $\mathbb{D}(X_z)$. Let $\overline{l}\in\overline{\mathbf{D}}/\varpi\overline{\mathbf{D}}$ be the image of $l$. Let $\overline{x_{\textup{crys},z}}$ be the image of $x_{\textup{crys},z}$ in $\mathbf{D}/\varpi\mathbf{D}$. The Hodge filtration of $X_{z}$ is given by $0\subset\mathbb{F}\cdot\overline{l}\oplus\{\overline{l}\}^{\perp}\subset\overline{\mathbf{D}}/\varpi\overline{\mathbf{D}}\oplus\mathbf{D}/\varpi\mathbf{D}$, where $\{\overline{l}\}^{\perp}\subset\mathbf{D}/\varpi\mathbf{D}$ is the kernel of the linear map $\langle\overline{l},\cdot\rangle:\mathbf{D}/\varpi\mathbf{D}\rightarrow\mathbb{F}$. The condition $t\geq1$ implies that $\overline{x_{\textup{crys},z}}\in\{\overline{l}\}^{\perp}$. Therefore $\Phi(x_{\textup{crys},z})\in\varpi\mathbb{D}(X_z)$. Then $\mathbb{D}(x)(\overline{1}_1)=\varpi^{-1}\mathbb{D}(x)(\Phi\overline{1}_0)=\varpi^{-1}\Phi x_{\textup{crys},z}\in\mathbb{D}(X_z)$. Therefore $\mathbb{D}(x)$ maps $\mathbb{D}(\overline{\mathbb{Y}})$ to $\mathbb{D}(X_z)$, i.e., $x$ is a homomorphism by the theory of Dieudonne modules. Therefore $z\in\mathcal{Z}(x)(\mathbb{F})$, which contradicts to the assumption. Hence $t=0$.
    \par
    If $m=\infty$, i.e., the point $\Tilde{z}:\textup{Spf}\,\ofb\rightarrow\widehat{\mathcal{N}}^{u}_{z}$ factors through $\widehat{\mathcal{Z}(x)}_{z}$. If $t$ is a finite number. The morphism $\ofb/(\varpi^{t+1})\rightarrow\mathbb{F}$ is a surjection in $\textup{Art}_{\ofb}$ whose kernel admits a nilpotent $\mathcal{O}_{F}$-pd-structure. The line $\ofb/(\varpi^{t+1})\cdot l$ corresponds to the point $\Tilde{z}_{t+1}:\textup{Spec}\,\ofb/(\varpi^{t+1})\rightarrow\textup{Spf}\,\ofb\stackrel{\Tilde{z}}\rightarrow\widehat{\mathcal{N}}^{u}_{z}$ under the bijection (i) in Lemma \ref{deformu}. It also factors through $\widehat{\mathcal{Z}(x)}_{z}$. By definition, the image of $x_{\textup{crys},z}$ in $\mathbf{D}/\varpi^{t+1}\mathbf{D}$ equals to $x_{\textup{crys},z}(\ofb/(\varpi^{t+1}))$. Therefore $\langle x_{\textup{crys},z}(\ofb/(\varpi^{t+1})),l \rangle\equiv\langle x_{\textup{crys},z},l \rangle\equiv0\,\,\textup{mod}\,\varpi^{t+1}$ by Lemma \ref{deformu}. This contradicts to the fact that $\nu_{\varpi}(\langle x_{\textup{crys},z},l\rangle)=t$. Hence $t=\infty=m$.
    \par
    If $1\leq m<\infty$. The morphism $\ofb/(\varpi^{m})\rightarrow\mathbb{F}$ is a surjection in $\textup{Art}_{\ofb}$ whose kernel admits a nilpotent $\mathcal{O}_{F}$-pd-structure. The line $\ofb/(\varpi^{m})\cdot l$ corresponds to the point $\Tilde{z}_{m}:\textup{Spec}\,\ofb/(\varpi^{m})\rightarrow\textup{Spf}\,\ofb\stackrel{\Tilde{z}}\rightarrow\widehat{\mathcal{N}}^{u}_{z}$ under the bijection (i) in Lemma \ref{deformu}, which factors through $\widehat{\mathcal{Z}(x)}_{z}$. By definition, the image of $x_{\textup{crys},z}$ in $\mathbf{D}/\varpi^{m}\mathbf{D}$ equals to $x_{\textup{crys},z}(\ofb/(\varpi^{m}))$. Therefore $\langle x_{\textup{crys},z}(\ofb/(\varpi^{m})),l \rangle\equiv\langle x_{\textup{crys},z},l \rangle\equiv0\,\,\textup{mod}\,\varpi^{m}$ by Lemma \ref{deformu}. Hence $t\geq m$.
    \par
    On the other hand, the morphism $\ofb/(\varpi^{t})\rightarrow\mathbb{F}$ is a surjection in $\textup{Art}_{\ofb}$ whose kernel admits a nilpotent $\mathcal{O}_{F}$-pd-structure. The line $\ofb/(\varpi)^{t}\cdot l$ corresponds to the point $\Tilde{z}_{t}:\textup{Spec}\,\ofb/(\varpi^{t})\rightarrow\textup{Spf}\,\ofb\stackrel{\Tilde{z}}\rightarrow\widehat{\mathcal{N}}^{u}_{z}$ under the bijection (i) in Lemma \ref{deformu}. The point $\tilde{z}_{t}\in\widehat{\mathcal{N}}^{u}_{z}(\ofb/(\varpi^{t}))$ factors through $\widehat{\mathcal{Z}(x)}_{z}$ because $\langle x_{\textup{crys},z}(\ofb/(\varpi^{t})),l\rangle\equiv\langle x_{\textup{crys},z},l\rangle\equiv 0\,\,(\textup{mod}\,\,\varpi^{t})$ by Lemma \ref{deformu}. Therefore we have a morphism $\textup{Spec}\,\ofb/(\varpi^{t})\rightarrow\widehat{\mathcal{Z}(x)}_{z}\times_{\widehat{\mathcal{N}}^{u}_{z}}\textup{Spf}\,\ofb\simeq\textup{Spec}\,\ofb/(\varpi^{m})$ by the universal property of fiber product, hence $m\geq t$. Therefore we conclude that $m=t$.\\
    \par
    Now we consider the GSpin case, the proof is similar. Let $\widehat{\mathcal{Z}(x)}_{z}$ be the completion of the special cycle $\mathcal{Z}(x)$ at $z$, there exists a Cartesian diagram,
    \begin{equation*}
        \xymatrix{
    \textup{Spf}\,\zpb/(\Tilde{z}^{\sharp}(f_{x,z}))\ar[r]\ar[d]&\textup{Spf}\,\zpb\ar[d]^{\Tilde{z}}\\
    \widehat{\mathcal{Z}(x)}_{z}\ar[r]&\widehat{\mathcal{N}}^{o}_{z}.}
    \end{equation*}
    Let $m=\nu_{p}(\Tilde{z}^{\sharp}(f_{x,z}))\in\mathbb{Z}_{\geq0}\cup\{\infty\}$, $t=\nu_{p}(\langle x_{\textup{crys},z},l\rangle)\in\mathbb{Z}_{\geq0}\cup\{\infty\}$, the equality in (\ref{ordero}) is equivalent to $t= m$.
    \par
    We first show that the number $m$ is always nonzero. If $m=0$, we have $z\notin\mathcal{Z}(x)(\mathbb{F})$ by similar arguments in the unitary case. By Lemma \ref{gspin-simple-k-points}, the assumption $x_{\textup{crys},z}\in\mathbf{V}$ implies that $z\in\mathcal{Z}(x)(\mathbb{F})$, which is a contradiction. Hence $m$ is nonzero.
    \par
    If $m=\infty$, i.e., the point $\Tilde{z}:\textup{Spf}\,\zpb\rightarrow\widehat{\mathcal{N}}^{o}_{z}$ factors through $\widehat{\mathcal{Z}(x)}_{z}$. If $t$ is a finite number. The morphism $\zpb/(p^{t+1})\rightarrow\mathbb{F}$ is a surjection in $\textup{Art}_{\zpb}$ whose kernel admits a nilpotent divided power structure. The isotropic line $\zpb/(p^{t+1})\cdot l$ corresponds to the point $\Tilde{z}_{t+1}:\textup{Spec}\,\zpb/(p^{t+1})\rightarrow\textup{Spf}\,\zpb\stackrel{\Tilde{z}}\rightarrow\widehat{\mathcal{N}}^{o}_{z}$ under the bijection (i) in Lemma \ref{deformo}. It also factors through $\widehat{\mathcal{Z}(x)}_{z}$. By definition, the image of $x_{\textup{crys},z}$ in $\mathbf{V}/p^{t+1}\mathbf{V}$ equals to $x_{\textup{crys},z}(\zpb/(p^{t+1}))$. Therefore $( x_{\textup{crys},z}(\zpb/(p^{t+1})),l)\equiv( x_{\textup{crys},z},l )\equiv0\,\,\textup{mod}\,p^{t+1}$ by Lemma \ref{deformo}. This contradicts to the fact that $\nu_{p}(( x_{\textup{crys},z},l))=t$. Hence $t=\infty=m$.
    \par
    If $m<\infty$. The morphism $\zpb/(p^{m})\rightarrow\mathbb{F}$ is a surjection in $\textup{Art}_{\zpb}$ whose kernel admits a nilpotent divided power structure. The isotropic line $\zpb/(p^{m})\cdot l$ corresponds to the point $\Tilde{z}_{m}:\textup{Spec}\,\zpb/(p^{m})\rightarrow\textup{Spf}\,\zpb\stackrel{\Tilde{z}}\rightarrow\widehat{\mathcal{N}}^{o}_{z}$ under the bijection (i) in Lemma \ref{deformo}, which factors through $\widehat{\mathcal{Z}(x)}_{z}$. By definition, the image of $x_{\textup{crys},z}$ in $\mathbf{V}/p^{m}\mathbf{V}$ equals to $x_{\textup{crys},z}(\zpb/(p^{m}))$. Therefore $( x_{\textup{crys},z}(\zpb/(p^{m})),l )\equiv( x_{\textup{crys},z},l )\equiv0\,\,\textup{mod}\,p^{m}$ by Lemma \ref{deformo}. Hence $t\geq m$.
    \par
    On the other hand, the morphism $\zpb/(p^{t})\rightarrow\mathbb{F}$ is a surjection in $\textup{Art}_{\zpb}$ whose kernel admits a nilpotent divided power structure. The isotropic line $\zpb/(p)^{t}\cdot l$ corresponds to the point $\Tilde{z}_{t}:\textup{Spec}\,\zpb/(p^{t})\rightarrow\textup{Spf}\,\zpb\stackrel{\Tilde{z}}\rightarrow\widehat{\mathcal{N}}^{u}_{z}$ under the bijection (i) in Lemma \ref{deformo}. The point $\tilde{z}_{t}\in\widehat{\mathcal{N}}^{u}_{z}(\zpb/(p^{t}))$ factors through $\widehat{\mathcal{Z}(x)}_{z}$ because $( x_{\textup{crys},z}(\zpb/(p^{t})),l)\equiv( x_{\textup{crys},z},l)\equiv 0\,\,(\textup{mod}\,\,p^{t})$ by Lemma \ref{deformu}. Therefore we have a morphism $\textup{Spec}\,\zpb/(p^{t})\rightarrow\widehat{\mathcal{Z}(x)}_{z}\times_{\widehat{\mathcal{N}}^{u}_{z}}\textup{Spf}\,\zpb\simeq\textup{Spec}\,\zpb/(p^{m})$ by the universal property of fiber product, hence $m\geq t$. Therefore we conclude that $m=t$.\\
\end{proof}
\subsection{Comparison of several quantities}
\begin{definition}
    Let $x\in\mathbb{V}^{u}$ (resp. $x\in\mathbb{V}^{o}$) be a nonzero special quasi-homomorphism such that $(x,x)\in\mathcal{O}_{E}$ (resp. $q_{\mathbb{V}^{o}}(x)\in \zp$). Let $z\in\mathcal{Z}(x)(\mathbb{F})$ be a point. Let $\widehat{\mathcal{N}}_z^{u}$ (resp. $\widehat{\mathcal{N}}_z^{o}$) be the formal completion of $\mathcal{N}^{u}$ (resp. $\mathcal{N}^{o}$) at $z$. Let $\widehat{\mathcal{Z}(x)}_z$ be the formal completion of $\mathcal{Z}(x)$ at $z$. Define
    \begin{align*}
    l_{z,x}&=\textup{the largest number $l$ such that $z\in\mathcal{Z}(\varpi^{-l+1}x)(\mathbb{F})$ (resp. $z\in\mathcal{Z}(p^{-l+1}x)(\mathbb{F})$).}\\
    m_{z,x}&=\textup{the largest number $m$ such that $x_{\textup{crys},z}\in \varpi^{m}\mathbf{D}$ (resp. $x_{\textup{crys},z}\in p^{m}\mathbf{V}$).}\\
    n_{z,x}&=  \textup{the largest number $n$ (could be infinity) such that there exists a point } \\
        &\,\,\,\,\,\,\,\,\,\,\textup{$z^{\prime}\in\widehat{\mathcal{Z}(x)}_z(\ofb/(\varpi^{n}))$ (resp. $z^{\prime}\in\widehat{\mathcal{Z}(x)}_z(\zpb/(p^{n}))$).}\\        
    \end{align*}
    For a point $\Tilde{z}\in\widehat{\mathcal{N}}_z^{u}(\ofb)$ (resp. $\Tilde{z}\in\widehat{\mathcal{N}}_z^{o}(\zpb)$), let $l\in\overline{\mathbf{D}}$ (resp. $l\in\mathbf{V}$) be a generator of the line (resp. isotropic line) $L^{1}\subset\overline{\mathbf{D}}$ (resp. $L^{1}\subset\mathbf{V}$) corresponding to the point $\Tilde{z}$. Define
    \begin{align*}
        t_{x}(\Tilde{z})=\nu_{\varpi}(\langle x_{\textup{crys},z},l\rangle)\,\,\textup{(resp. $\nu_p((x_{\textup{crys},z},l))$)}.
    \end{align*}
\end{definition}
\begin{remark}
    In the set up of the above definition, the smoothness criterion (Lemma \ref{smoothcycleu} and Lemma \ref{smoothcycleo}) of the formal scheme $\mathcal{Z}(x)$ at a point $z\in\mathcal{Z}(x)(\mathbb{F})$ can be interpreted in the following way:
    \begin{itemize}
        \item[$\bullet$] In the unitary case, the formal scheme $\mathcal{Z}(x)$ is formally smooth over $\ofb$ of relative dimension $n-2$ at the point $z$ if and only if $m_{z,x}=0$.
        \item[$\bullet$] In the GSpin case, the formal scheme $\mathcal{Z}(x)$ is formally smooth over $\zpb$ of relative dimension $n-2$ at the point $z$ if and only if $m_{z,x}=0$ and $n_{z,x}\geq2$.
    \end{itemize}
    \label{smooth-new}
\end{remark}
\begin{proposition}
    Let $x\in\mathbb{V}^{u}$ (resp. $x\in\mathbb{V}^{o}$) be a nonzero special quasi-homomorphism such that $(x,x)\in\mathcal{O}_{E}$ (resp. $q_{\mathbb{V}^{o}}(x)\in \zp$). Let $z\in\mathcal{Z}(x)(\mathbb{F})$ be a point. We have
    \begin{align*}
    l_{z,x}&=\min\limits_{\Tilde{z}\in\widehat{\mathcal{N}}_z^{u}(\ofb)}\{t_x(\Tilde{z})\}\,\,\textup{(resp. $\min\limits_{\Tilde{z}\in\widehat{\mathcal{N}}_z^{o}(\ofb)}\{t_x(\Tilde{z})\}$)}.\\
        n_{z,x}&=\max\limits_{\Tilde{z}\in\widehat{\mathcal{N}}_z^{u}(\ofb)}\{t_x(\Tilde{z})\}\,\,\textup{(resp. $\max\limits_{\Tilde{z}\in\widehat{\mathcal{N}}_z^{o}(\ofb)}\{t_x(\Tilde{z})\}$)}.
    \end{align*}
    Moreover, we have $l_{z,x}=m_{z,x}+1$ in the GSpin case, while $l_{z,x}-m_{z,x}$ can be $0$ or $1$ in the unitary case. 
    \label{comparison}
\end{proposition}
\begin{proof}
We first prove the equality about $n_{z,x}$. Then we prove the equality about $l_{z,x}$ while which the statements about the relationships between $l_{z,x}$ and $m_{z,x}$ will be clear. We separate the unitary and the GSpin case. 
\par
In the unitary case. For a point $\Tilde{z}\in\widehat{\mathcal{N}}_z^{u}(\ofb)$, denote by $l_{\tilde{z}}\in\overline{\mathbf{D}}$ be a generator of the line $L^{1}\subset\overline{\mathbf{D}}$ corresponding to the point $\Tilde{z}$. For an integer $n$, let $f_{\varpi^{n}x,z}\in\widehat{\mathcal{O}}_{\mathcal{N}^{u},z}$ be the local equation of the special divisor $\mathcal{Z}(\varpi^{n}x)$ at $z$. By Proposition \ref{orderlemma}, we have 
\begin{equation*}
    (\Tilde{z}^{\sharp}(f_{\varpi^{n}x}))=(\langle \varpi^{n} x_{\textup{crys},z},l_{\tilde{z}}\rangle)=(\varpi^{n+t_{x}(\Tilde{z})})\subset(\varpi)
\end{equation*}
as long as $z\in\mathcal{Z}(\varpi^{n}x)(\mathbb{F})$. Notice that the point $z\in\mathcal{Z}(\varpi^{-l_{z,x}+1}x)(\mathbb{F})$. Let $n=-l_{z,x}+1$, we get $-l_{z,x}+1+t_{x}(\Tilde{z})\geq1$, i.e., $l_{z,x}\leq t_{x}(\Tilde{z})$. Let $n=0$, the fiber product $\widehat{\mathcal{Z}(x)}_{z}\times_{\widehat{\mathcal{N}}^{u}_{z},\Tilde{z}}\textup{Spf}\,\ofb\simeq\textup{Spec}\,\ofb/(\varpi^{t_{x}(\Tilde{z})})$ defines a point $z^{\prime}\in \widehat{\mathcal{Z}(x)}_{z}(\varpi^{t_{x}(\Tilde{z})})$, hence $n_{z,x}\geq t_x(\Tilde{z})$. Therefore we have proved that $l_{z,x}\leq t_{x}(\Tilde{z})\leq n_{z,x}$ for all points $\Tilde{z}\in\widehat{\mathcal{N}}_z^{u}(\ofb)$. Hence
\begin{equation*}
    l_{z,x}\leq\min\limits_{\Tilde{z}\in\widehat{\mathcal{N}}_z^{u}(\ofb)}\{t_x(\Tilde{z})\}\leq\max\limits_{\Tilde{z}\in\widehat{\mathcal{N}}_z^{u}(\ofb)}\{t_x(\Tilde{z})\}\leq n_{z,x}.
\end{equation*}
\par
We first prove $n_{z,x}=\max\limits_{\Tilde{z}\in\widehat{\mathcal{N}}_z^{u}(\ofb)}\{t_x(\Tilde{z})\}$. Let $M=\max\limits_{\Tilde{z}\in\widehat{\mathcal{N}}_z^{u}(\ofb)}\{t_x(\Tilde{z})\}$. If $n_{z,x}=\infty$. For all integer $n$, there exists a point $z^{\prime}\in\widehat{\mathcal{Z}(x)}_z(\ofb/(\varpi^{n}))$. Let $\Tilde{z}\in\widehat{\mathcal{N}}_z^{u}(\ofb)$ be a lift of $z^{\prime}$ to $\ofb$, the existence of $\Tilde{z}$ follows from the fact that $\widehat{\mathcal{O}}_{\mathcal{N}^{u},z}$ is formally smooth over $\ofb$. Then we have $\langle l_{\Tilde{z}}, x_{\textup{crys},z}\rangle\equiv 0$ mod $\varpi^{n}$ by Lemma \ref{deformu}. Hence $t_{x}(\Tilde{z})=\nu_{\varpi}(\langle l_{\tilde{z}}, x_{\textup{crys},z}\rangle)\geq n$. Therefore $M=\infty=n_{z,x}$. If $n_{z,x}<\infty$, there exists a point $z^{\prime}\in\widehat{\mathcal{Z}(x)}_z(\ofb/(\varpi^{n_{z,x}}))$. Let $\Tilde{z}\in\widehat{\mathcal{N}}_z^{u}(\ofb)$ be a lift of $z^{\prime}$ to $\ofb$. Then we have $\langle l_{\Tilde{z}}, x_{\textup{crys},z}\rangle\equiv 0$ mod $\varpi^{n_{z,x}}$ by Lemma \ref{deformu}. Hence $M\geq n_{z,x}$. We already have $M\leq n_{z,x}$. Hence $n_{z,x}=M=\max\limits_{\Tilde{z}\in\widehat{\mathcal{N}}_z^{u}(\ofb)}\{t_x(\Tilde{z})\}$.
\par
Next we prove $l_{z,x}=\min\limits_{\Tilde{z}\in\widehat{\mathcal{N}}_z^{u}(\ofb)}\{t_x(\Tilde{z})\}$. Let $N=\min\limits_{\Tilde{z}\in\widehat{\mathcal{N}}_z^{u}(\ofb)}\{t_x(\Tilde{z})\}$. Let $x^{\prime}=\varpi^{-l_{z,x}+1}x$. Then we have $z\in\mathcal{Z}(x^{\prime})(\mathbb{F})$ but $z\notin\mathcal{Z}(\varpi^{-1}x^{\prime})(\mathbb{F})$. There are two possibilities.
\begin{itemize}
    \item[(a)] The element $\varpi^{-1}x_{\textup{crys},z}^{\prime}\notin\mathbf{D}$. Let $\Tilde{z}\in\widehat{\mathcal{N}}_z^{u}(\ofb)$ be a point. Notice that $t_{x^{\prime}}(\Tilde{z})\geq1$ since $z\in\mathcal{Z}(x^{\prime})(\mathbb{F})$. If $t_{x^{\prime}}(\Tilde{z})=1$, then $t_{x}(\Tilde{z})=t_{x^{\prime}}(\Tilde{z})+l_{z,x}-1=l_{z,x}$, hence $l_{z,x}\geq N$. If $t_{x^{\prime}}(\Tilde{z})\geq2$, let $l\in\overline{\mathbf{D}}$ be an element such that $\nu_{\varpi}(\langle l,x_{\textup{crys},z}\rangle)=0$, the existence of $l$ follows from the fact that $x_{\textup{crys},z}\notin\varpi\mathbf{D}$ and the pairing $\langle\cdot,\cdot\rangle:\overline{\mathbf{D}}\times\mathbf{D}\rightarrow\ofb$ is perfect. Let $l^{\prime}=l_{\Tilde{z}}+\varpi\cdot l$. Then $\langle l^{\prime}, x_{\textup{crys},z}\rangle=\langle l_{\tilde{z}},x_{\textup{crys},z}\rangle+\varpi\langle l,x_{\textup{crys},z}\rangle$. Hence $\nu_{\varpi}(\langle l^{\prime}, x_{\textup{crys},z}\rangle)=1$. Let $\Tilde{z}^{\prime}\in\widehat{\mathcal{N}}_z^{u}(\ofb)$ be the point corresponding to the line generated by $l^{\prime}$ in Lemma \ref{deformu}. Then $t_{x^{\prime}}(\Tilde{z}^{\prime})=1$, then $t_{x}(\Tilde{z}^{\prime})=t_{x^{\prime}}(\Tilde{z}^{\prime})+l_{z,x}-1=l_{z,x}$, hence $l_{z,x}\geq N$.
\par
Notice that in this case $m_{z,x^{\prime}}=0$, therefore $m_{z,x}=m_{z,x^{\prime}}+l_{z,x}-1=l_{z,x}-1$. We give an example of case (a) in Example \ref{difference1}.
\item[(b)] The element $\varpi^{-1}x_{\textup{crys},z}^{\prime}\in\mathbf{D}$. Let $\Tilde{z}\in\widehat{\mathcal{N}}_z^{u}(\ofb)$ be a point, then $\nu_{\varpi}(\langle l_{\Tilde{z}}, \varpi^{-1}x^{\prime}_{\textup{crys},z}\rangle)=0$ because otherwise $z\in\mathcal{Z}(\varpi^{-1}x^{\prime})(\mathbb{F})$ by Proposition \ref{orderlemma}. Hence $t_{x^{\prime}}(\Tilde{z}^{\prime})=1$, then $t_{x}(\Tilde{z}^{\prime})=t_{x^{\prime}}(\Tilde{z}^{\prime})+l_{z,x}-1=l_{z,x}$, hence $l_{z,x}\geq N$.
\par
In this case we can conclude that $l_{z,x}=m_{z,x}$: the assumption $\varpi^{-1}x_{\textup{crys},z}^{\prime}\in\mathbf{D}$ implies that $m_{z,x^{\prime}}=m_{z,x}-l_{z,x}+1\geq1$, hence $m_{z,x}\geq l_{z,x}$. If $m_{z,x}\geq l_{z,x}\geq1$, we have $\varpi^{-1}x_{\textup{crys},z}^{\prime}\in\varpi\mathbf{D}$. Then $\nu_{\varpi}(\langle l_{\Tilde{z}}, \varpi^{-1}x^{\prime}_{\textup{crys},z}\rangle)\geq1$. Proposition \ref{orderlemma} implies that $z\in\mathcal{Z}(\varpi^{-1}x^{\prime})(\mathbb{F})$, which is a contradiction. Therefore $l_{z,x}=m_{z,x}$. We give an example of case (b) in Example \ref{difference1}.
\end{itemize}
\par
In both cases, we have $l_{z,x}\geq N$. We already have $N\geq l_{z,x}$, hence $l_{z,x}=N=\min\limits_{\Tilde{z}\in\widehat{\mathcal{N}}_z^{u}(\ofb)}\{t_x(\Tilde{z})\}$. Notice that $l_{z,x}-m_{z,x}$ can be $0$ or $1$.
\\
\par
The proof for the GSpin case is basically the same. For a point $\Tilde{z}\in\widehat{\mathcal{N}}_z^{o}(\zpb)$, denote by $l_{\tilde{z}}\in\mathbf{V}$ be a generator of the line $L^{1}\subset\mathbf{V}$ corresponding to the point $\Tilde{z}$. For an integer $n$, let $f_{p^{n}x,z}\in\widehat{\mathcal{O}}_{\mathcal{N}^{o},z}$ be the local equation of the special divisor $\mathcal{Z}(p^{n}x)$ at $z$. By Proposition \ref{orderlemma}, we have 
\begin{equation*}
    (\Tilde{z}^{\sharp}(f_{p^{n}x}))=(( p^{n} x_{\textup{crys},z},l_{\tilde{z}}))=(p^{n+t_{x}(\Tilde{z})})\subset(p)
\end{equation*}
as long as $z\in\mathcal{Z}(p^{n}x)(\mathbb{F})$. Notice that the point $z\in\mathcal{Z}(p^{-l_{z,x}+1}x)(\mathbb{F})$. Let $n=-l_{z,x}+1$, we get $-l_{z,x}+1+t_{x}(\Tilde{z})\geq1$, i.e., $l_{z,x}\leq t_{x}(\Tilde{z})$. Let $n=0$, the fiber product $\widehat{\mathcal{Z}(x)}_{z}\times_{\widehat{\mathcal{N}}^{o}_{z},\Tilde{z}}\textup{Spf}\,\zpb\simeq\textup{Spec}\,\zpb/(p^{t_{x}(\Tilde{z})})$ defines a point $z^{\prime}\in \widehat{\mathcal{Z}(x)}_{z}(p^{t_{x}(\Tilde{z})})$, hence $n_{z,x}\geq t_x(\Tilde{z})$. Therefore we have proved that $l_{z,x}\leq t_{x}(\Tilde{z})\leq n_{z,x}$ for all points $\Tilde{z}\in\widehat{\mathcal{N}}_z^{o}(\zpb)$. Hence
\begin{equation*}
    l_{z,x}\leq\min\limits_{\Tilde{z}\in\widehat{\mathcal{N}}_z^{o}(\zpb)}\{t_x(\Tilde{z})\}\leq\max\limits_{\Tilde{z}\in\widehat{\mathcal{N}}_z^{o}(\zpb)}\{t_x(\Tilde{z})\}\leq n_{z,x}.
\end{equation*}
\par
We first prove $n_{z,x}=\max\limits_{\Tilde{z}\in\widehat{\mathcal{N}}_z^{o}(\zpb)}\{t_x(\Tilde{z})\}$. Let $M=\max\limits_{\Tilde{z}\in\widehat{\mathcal{N}}_z^{o}(\zpb)}\{t_x(\Tilde{z})\}$. If $n_{z,x}=\infty$. For all integers $n$, there exists a point $z^{\prime}\in\widehat{\mathcal{Z}(x)}_z(\zpb/(p^{n}))$. Let $\Tilde{z}\in\widehat{\mathcal{N}}_z^{o}(\zpb)$ be a lift of $z^{\prime}$ to $\zpb$, the existence of $\Tilde{z}$ follows from the fact that $\widehat{\mathcal{O}}_{\mathcal{N}^{o},z}$ is formally smooth over $\zpb$. Then we have $( l_{\Tilde{z}}, x_{\textup{crys},z})\equiv 0$ mod $p^{n}$ by Lemma \ref{deformu}. Hence $t_{x}(\Tilde{z})=\nu_{p}(( l_{\tilde{z}}, x_{\textup{crys},z}))\geq n$. Therefore $M=\infty=n_{z,x}$. If $n_{z,x}<\infty$, there exists a point $z^{\prime}\in\widehat{\mathcal{Z}(x)}_z(\zpb/(p^{n_{z,x}}))$. Let $\Tilde{z}\in\widehat{\mathcal{N}}_z^{o}(\zpb)$ be a lift of $z^{\prime}$ to $\zpb$. Then we have $( l_{\Tilde{z}}, x_{\textup{crys},z})\equiv 0$ mod $p^{n_{z,x}}$ by Lemma \ref{deformu}. Hence $M\geq n_{z,x}$. We already have $M\leq n_{z,x}$. Hence $n_{z,x}=M=\max\limits_{\Tilde{z}\in\widehat{\mathcal{N}}_z^{o}(\zpb)}\{t_x(\Tilde{z})\}$.
\par
Next we prove $l_{z,x}=\min\limits_{\Tilde{z}\in\widehat{\mathcal{N}}_z^{o}(\zpb)}\{t_x(\Tilde{z})\}$. Let $N=\min\limits_{\Tilde{z}\in\widehat{\mathcal{N}}_z^{o}(\zpb)}\{t_x(\Tilde{z})\}$. Let $x^{\prime}=p^{-l_{z,x}+1}x$. Then we have $z\in\mathcal{Z}(x^{\prime})(\mathbb{F})$ but $z\notin\mathcal{Z}(p^{-1}x^{\prime})(\mathbb{F})$. We must have the element $p^{-1}x_{\textup{crys},z}^{\prime}\notin\mathbf{V}$ because otherwise $z\in\mathcal{Z}(p^{-1}x^{\prime})(\mathbb{F})$ by Lemma \ref{gspin-simple-k-points}.
\par
Let $\Tilde{z}\in\widehat{\mathcal{N}}_z^{o}(\zpb)$ be a point. Notice that $t_{x^{\prime}}(\Tilde{z})\geq1$ since $z\in\mathcal{Z}(x^{\prime})(\mathbb{F})$. If $t_{x^{\prime}}(\Tilde{z})=1$, then $t_{x}(\Tilde{z})=t_{x^{\prime}}(\Tilde{z})+l_{z,x}-1=l_{z,x}$, hence $l_{z,x}\geq N$. If $t_{x^{\prime}}(\Tilde{z})\geq2$, let $l\in\mathbf{V}$ be an element such that $\nu_{p}(( l,x_{\textup{crys},z}))=0$, the existence of $l$ follows from the fact that $x_{\textup{crys},z}\notin\mathbf{V}$ and the pairing $(\cdot,\cdot):\mathbf{V}\times\mathbf{V}\rightarrow\zpb$ is perfect. Let $l^{\prime}=l_{\Tilde{z}}+p\cdot l$. Then $( l^{\prime}, x_{\textup{crys},z})=( l_{\tilde{z}},x_{\textup{crys},z})+p(l,x_{\textup{crys},z})$. Hence $\nu_{p}(( l^{\prime}, x_{\textup{crys},z}))=1$. Let $\Tilde{z}^{\prime}\in\widehat{\mathcal{N}}_z^{o}(\zpb)$ be the point corresponding to the line generated by $l^{\prime}$ in Lemma \ref{deformo}. Then $t_{x^{\prime}}(\Tilde{z}^{\prime})=1$, then $t_{x}(\Tilde{z}^{\prime})=t_{x^{\prime}}(\Tilde{z}^{\prime})+l_{z,x}-1=l_{z,x}$, hence $l_{z,x}\geq N$. Since we already have $N\geq l_{z,x}$, we get $l_{z,x}=N=\min\limits_{\Tilde{z}\in\widehat{\mathcal{N}}_z^{o}(\zpb)}\{t_x(\Tilde{z})\}$. 
\par
Notice that $m_{z,x^{\prime}}=0$, therefore $m_{z,x}=m_{z,x^{\prime}}+l_{z,x}-1=l_{z,x}-1$.
\end{proof}
\begin{example}
    In the unitary case, the difference $l_{z,x}-m_{z,x}$ can be $0$ (case (b)) or $1$ (case (a)). We give examples for both cases.
    \begin{itemize}
        \item [(a)] Let $n=2$. Let $x$ be a special quasi-homomorphism such that $\nu_{\varpi}((x,x))=0$. Then $\mathcal{N}^{u}\simeq\textup{Spf}\,\ofb[[t]]$. We have $\mathcal{Z}(x)\simeq\simeq\textup{Spf}\,\ofb$ (cf. \cite[Proposition 8.1]{KR11}). It is formally smooth over $\ofb$, hence $m_{z,x}=0$ by Lemma \ref{smoothcycleu}. Therefore $\varpi^{-1}x_{\textup{crys},z}\notin\mathbf{D}$. We have $l_{z,x}=1$ and $m_{z,x}=0$
        \item[(b)] Let $n=2$. Let $x$ be a special quasi-homomorphism such that $\nu_{\varpi}((x,x))=1$. Then $\mathcal{N}^{u}\simeq\textup{Spf}\,\ofb[[t]]$. We have $\mathcal{Z}(x)\simeq\textup{Spf}\,W_1$ (cf. \cite[Proposition 8.1]{KR11}). Here $W_1$ is a ramified extension of $\ofb$ of degree $q+1$. It is not formally smooth over $\ofb$, hence $m_{z,x}\geq1$ by Lemma \ref{smoothcycleu}, i.e., $\varpi^{-1}x_{\textup{crys},z}\in\mathbf{D}$. We have $l_{z,x}=m_{z,x}=1$ by the above proof.
    \end{itemize}
    \label{difference1}
\end{example}
\begin{lemma}
    Let $x\in\mathbb{V}^{u}$ (resp. $x\in\mathbb{V}^{o}$) be a nonzero special quasi-homomorphism such that $(x,x)\in\mathcal{O}_{E}$ (resp. $q_{\mathbb{V}^{o}}(x)\in \zp$). Let $z\in\mathcal{Z}(x)(\mathbb{F})$ be a point. The formal scheme $\mathcal{Z}(x)$ is formally smooth over $\ofb$ (resp. $\zpb$) of relative dimension $n-2$ at the point $z$ if and only if $l_{z,x}=1$ and $n_{z,x}\geq2$.
    \label{smooth-lemma-2}
\end{lemma}
\begin{proof}
    We first prove the ``only if" direction: if the formal scheme $\mathcal{Z}(x)$ is formally smooth over $\ofb$ (resp. $\zpb$) of relative dimension $n-2$ at the point $z$. Then $n_{z,x}=\infty$ by the formally smoothness. We also have $m_{z,x}=0$ by Remark \ref{smooth-new}. By Proposition \ref{comparison}, we have $l_{z,x}=m_{z,x}+1=1$ in the orthogonal case. And $l_{z,x}=0$ or $1$ in the unitary case. However, since $z\in\mathcal{Z}(x)(\mathbb{F})$, we must have $l_{z,x}\geq1$. Therefore in both cases we have $l_{z,x}=1$.
    \par
    Now we prove the ``if" direction. By Remark \ref{smooth-new}, we only need to prove the two conditions $l_{z,x}=1$ and $n_{z,x}\geq2$ imply the equality $m_{z,x}=0$. Let's assume the converse, i.e., $m_{z,x}\geq1$. Therefore $\varpi^{-1}x_{\textup{crys},z}\in\mathbf{D}$ (resp. $p^{-1}x_{\textup{crys},z}\in\mathbf{V}$).
    \par
    Since $n_{z,x}\geq2$, there exists a point $\Tilde{z}\in\mathcal{N}^{u}(\ofb)$ (resp. $\Tilde{z}\in\mathcal{N}^{o}(\zpb)$) such that $t_{x}(\Tilde{z})\geq2$. Let $l\in\overline{\mathbf{D}}$ (resp. $l\in\mathbf{V}$) be a generator of the line $L^{1}\subset\overline{\mathbf{D}}$ (resp. the isotropic line $L^{1}\subset\mathbf{V}$) corresponding to the point $\Tilde{z}\in\mathcal{N}^{u}(\ofb)$ (resp.  $\Tilde{z}\in\mathcal{N}^{o}(\zpb)$), then $\nu_{\varpi}(\langle x_{\textup{crys},z},l\rangle)\geq2$ (resp. $\nu_p((x_{\textup{crys},z},l))\geq2$), hence $\nu_{\varpi}(\langle \varpi^{-1}x_{\textup{crys},z},l\rangle)\geq1$ (resp. $\nu_p((p^{-1}x_{\textup{crys},z},l))\geq1$).
    \par
    Let $f_{\varpi^{-1}x,z}\in\widehat{\mathcal{O}}_{\mathcal{N}^{u},z}$ (resp. $f_{p^{-1}x,z}\in\widehat{\mathcal{O}}_{\mathcal{N}^{o},z}$) be the local equation of $\mathcal{Z}(\varpi^{-1}x)$ (resp. $\mathcal{Z}(p^{-1}x)$) at the point $z$. By Lemma \ref{orderlemma}, the ideal $(\Tilde{z}^{\sharp}(f_{\varpi^{-1}x,z}))=(\langle\varpi^{-1}x_{\textup{crys},z},l\rangle)\subset\ofb$ (resp. $(\Tilde{z}^{\sharp}(f_{p^{-1}x,z}))=((p^{-1}x_{\textup{crys},z},l))\subset\zpb$), hence $f_{\varpi^{-1}x,z}\in\widehat{\mathcal{O}}_{\mathcal{N}^{u},z}$ is not a unit (resp. $f_{p^{-1}x,z}\in\widehat{\mathcal{O}}_{\mathcal{N}^{o},z}$ is not a unit). Therefore $z\in\mathcal{Z}(\varpi^{-1}x)(\mathbb{F})$ (resp. $z\in\mathcal{Z}(p^{-1}x)(\mathbb{F})$), i.e., $l_{z,x}\geq2$, which is a contradiction. Therefore we have $m_{z,x}=0$. Then the formal scheme $\mathcal{Z}(x)$ is formally smooth over $\ofb$ (resp. $\zpb$) of relative dimension $n-2$ at the point $z$ by Lemma \ref{smoothcycleu} (resp. the extra condition $n_{z,x}\geq2$ and Lemma \ref{smoothcycleo}).
\end{proof}
\subsection{Proof of the main results}
\begin{theorem}
Let $x\in\mathbb{V}^{u}$ (resp. $x\in\mathbb{V}^{o}$) be a nonzero special quasi-homomorphism such that $(x,x)\in\mathcal{O}_{E}$ (resp. $q_{\mathbb{V}^{o}}(x)\in \zp$).\begin{itemize}
    \item[(i)] There is an equality $\mathcal{D}(x)(\mathbb{F})=\mathcal{Z}(x)(\mathbb{F})$.
\end{itemize}
Let $z\in\mathcal{N}^{u}(\mathbb{F})$ (resp. $z\in\mathcal{N}^{o}(\mathbb{F})$) be a point such that $z\in\mathcal{Z}(x)(\mathbb{F})$ but $z\notin\mathcal{Z}(\varpi^{-1}x)(\mathbb{F})$ (resp. $z\in\mathcal{Z}(x)(\mathbb{F})$ but $z\notin\mathcal{Z}(p^{-1}x)(\mathbb{F})$).\begin{itemize}
    \item[(ii)] The point $z\in\mathcal{D}(\varpi^{a}x)(\mathbb{F})$ (resp. $z\in\mathcal{D}(p^{a}x)(\mathbb{F})$) for all $a\geq0$.
    \item[(iii)] For all $a\geq1$, the difference divisor $\mathcal{D}(\varpi^{a}x)$ (resp. $\mathcal{D}(p^{a}x)$) is regular but not formally smooth over $\ofb$ (resp. $\zpb$) at $z$.
    \item[(iv)] For $a=0$, the difference divisor $\mathcal{D}(x)$ is regular at $z$, it is formally smooth over $\ofb$ (resp. $\zpb$) at $z$ if and only if $n_{z,x}\geq2$.
\end{itemize}
\label{main-results}
\end{theorem}
\noindent\textit{Proof of the unitary case}. We first prove (ii), (iii) and (iv), then come back to prove (i).\par
Let $\mathfrak{m}_{z}=(\varpi,t_{1},\cdots,t_{n-1})$ be the maximal ideal of the complete local ring $\widehat{\mathcal{O}}_{\mathcal{N}^{u},z}$. Notice that the set $\widehat{\mathcal{N}}^{u}_z(\ofb)$ has a bijection to the set of $\ofb$-algebra homomorphisms $\ofb[[t_1,\cdots,t_{n-1}]]\rightarrow\ofb$ by sending a point $\Tilde{z}$ to $\ofb$-algebra homomorphism $\Tilde{z}^{\sharp}:\ofb[[t_1,\cdots,t_{n-1}]]\rightarrow\ofb$.
\par
For $a\geq0$, let $f_{\varpi^{a}x}\in\widehat{\mathcal{O}}_{\mathcal{N}^{u},z}$ be the local equation of the special divisor $\mathcal{Z}(\varpi^{a}x)$ at $z$, and let $d_{\varpi^{a}x}\in\widehat{\mathcal{O}}_{\mathcal{N}^{u},z}$ be the local equation of the difference divisor $\mathcal{D}(\varpi^{a}x)$ at $z$. Notice that the point $z\in\mathcal{D}(\varpi^{a}x)(\mathbb{F})$ if and only if $d_{\varpi^{a}x}\in\mathfrak{m}_z$. Since $z\notin\mathcal{Z}(\varpi^{-1}x)$, we can assume $f_x=d_x$.
\par
We first prove $d_x\in \mathfrak{m}_{z}\backslash\mathfrak{m}_{z}^{2}$, which is equivalent to $f_x\in \mathfrak{m}_{z}\backslash\mathfrak{m}_{z}^{2}$. There are two cases.
\begin{itemize}
    \item[$\bullet$] If $n_{z,x}\leq1$. Notice that $l_{z,x}=1$ since $z\notin\mathcal{Z}(\varpi^{-1}x)(\mathbb{F})$. By Proposition \ref{comparison}, $l_{z,x}\leq t_{x}(\Tilde{z})\leq n_{z,x}$ for all points $\Tilde{z}\in\widehat{\mathcal{N}}^{u}_z(\ofb)$, then we must have $l_{z,x}= t_{x}(\Tilde{z})= n_{z,x}=1$. By Lemma \ref{orderlemma}, the number $\nu_{\varpi}(\Tilde{z}^{\sharp}(f_{x}))=t_x(\Tilde{z})=1$ for all $\ofb$-algebra homomorphism $\Tilde{z}^{\sharp}:\ofb[[t_{1},\cdots,t_{n-1}]]\rightarrow \ofb$, hence $f_{x}\equiv (\textup{unit})\cdot\varpi \,\,(\textup{mod}\,\,\mathfrak{m}_{z}^{2})$ by Lemma \ref{tech1}, then $d_x=f_x\in \mathfrak{m}_{z}\backslash\mathfrak{m}_{z}^{2}$. Notice that in this case the difference divisor $\mathcal{D}(x)$ is regular but not formally smooth over $\ofb$ at $z$.
    \item[$\bullet$] If $n_{z,x}\geq2$. Notice that we already have $l_{z,x}=1$, then by Lemma \ref{smooth-lemma-2} the formal scheme $\mathcal{Z}(x)$ is smooth over $\ofb$ of relative dimension $n-2$ at the point $z$, hence $f_{x}\in \mathfrak{m}_{z}\backslash\mathfrak{m}_{z}^{2}$. Notice that in this case the difference divisor $\mathcal{D}(x)$ is formally smooth over $\ofb$ at $z$.
\end{itemize}
\par
Now we already show that $d_x\in \mathfrak{m}_{z}\backslash\mathfrak{m}_{z}^{2}$. Our next goal is to prove $d_{\varpi^{a}x}\equiv(\textup{unit})\cdot\varpi\,\textup{mod}\,\mathfrak{m}_{z}^{2}$ for all $a\geq1$. Let $\Tilde{z}\in\widehat{\mathcal{N}}^{u}_z(\ofb)$ be a point, by Lemma \ref{orderlemma},
\begin{equation}
    (\Tilde{z}^{\sharp}(f_{\varpi^{a}x}))=(\langle\varpi^{a}\cdot x_{\textup{crys,z}},l\rangle)=(\varpi^{a}\cdot\langle x_{\textup{crys,z}},l\rangle)=(\varpi^{a}\cdot\Tilde{z}^{\sharp}(f_{x})).
    \label{key-cartesian}
\end{equation}
Let $\widehat{\mathcal{Z}(\varpi^{a}x)}_{z}$ be the completion of the formal scheme $\mathcal{Z}(\varpi^{a}x)$ at $z$, there exists a Cartesian diagram by (\ref{key-cartesian}),
\begin{equation*}
     \xymatrix{
    \textup{Spf}\,\ofb/(\varpi^{a}\cdot\Tilde{z}^{\sharp}(f_{x}))\ar[r]\ar[d]&\textup{Spf}\,\ofb\ar[d]^{\tilde{z}}\\
    \widehat{\mathcal{Z}(\varpi^{a}x)}_{z}\ar[r]&\widehat{\mathcal{N}}^{u}_{z}.}
\end{equation*}
Recall that by the definition of difference divisors, we have $f_{\varpi^{a}x}=\prod\limits_{i=0}^{a}d_{\varpi^{i}x}$ for all integer $a\geq0$, hence Lemma \ref{technical} implies that $d_{\varpi^{a}x}\equiv(\textup{unit})\cdot \varpi \,\,(\textup{mod}\,\,\mathfrak{m}_{z}^{2})$ for all $a\geq1$ since $d_{x}\in \mathfrak{m}_{z}\backslash\mathfrak{m}_{z}^{2}$.\par
For all integer $a\geq0$, the point $z\in\mathcal{D}(\varpi^{a}x)(\mathbb{F})$ because the local equation $d_{\varpi^{a}x}\in\mathfrak{m}_z$. Moreover, the local equation $d_{x}\in \mathfrak{m}_{z}\backslash\mathfrak{m}_{z}^{2}$ implies that the difference divisor $\mathcal{D}(\varpi^{a}x)$ is regular at the point $z$.
\par
For $a=0$, we have proved that the difference divisor $\mathcal{D}(x)$ is formally smooth over $\ofb$ at $z$ if and only if $n_{z,x}\geq2$.
For all $a\geq1$, the equation $d_{\varpi^{a}x}\equiv(\textup{unit})\cdot \varpi \,\,(\textup{mod}\,\,\mathfrak{m}_{z}^{2})$, hence the difference divisor $\mathcal{D}(\varpi^{a}x)$ is not formally smooth at the point $z$.
\par
Now we have shown (ii), (iii) and (iv), let's come back to prove (i). The one-side inclusion $\mathcal{D}(x)(\mathbb{F})\subset\mathcal{Z}(x)(\mathbb{F})$ follows from definition. Conversely, if $z^{\prime}\in\mathcal{Z}(x)(\mathbb{F})$. Let $x^{\prime}=\varpi^{-l_{z^{\prime},x}+1}x$, then $z^{\prime}\in\mathcal{Z}(x^{\prime})(\mathbb{F})$ but $z^{\prime}\notin\mathcal{Z}(\varpi^{-1}x^{\prime})(\mathbb{F})$. (ii) implies that $z^{\prime}\in\mathcal{D}(\varpi^{l_{z^{\prime},x}-1}x^{\prime})(\mathbb{F})=\mathcal{D}(x)(\mathbb{F})$. Therefore we conclude that $\mathcal{D}(x)(\mathbb{F})=\mathcal{Z}(x)(\mathbb{F})$.\hfill$\qed$\\\par
\noindent\textit{Proof of the GSpin case}. The proof is similar to the unitary case above. We first prove (ii), (iii) and (iv), then come back to prove (i).\par
Let $\mathfrak{m}_{z}=(p,t_{1},\cdots,t_{n-1})$ be the maximal ideal of the complete local ring $\widehat{\mathcal{O}}_{\mathcal{N}^{o},z}$. Notice that the set $\widehat{\mathcal{N}}^{o}_z(\zpb)$ has a bijection to the set of $\zpb$-algebra homomorphisms $\zpb[[t_1,\cdots,t_{n-1}]]\rightarrow\zpb$ by sending a point $\Tilde{z}$ to $\zpb$-algebra homomorphism $\Tilde{z}^{\sharp}:\zpb[[t_1,\cdots,t_{n-1}]]\rightarrow\zpb$.
\par
For $a\geq0$, let $f_{p^{a}x}\in\widehat{\mathcal{O}}_{\mathcal{N}^{o},z}$ be the local equation of the special divisor $\mathcal{Z}(p^{a}x)$ at $z$, and let $d_{p^{a}x}\in\widehat{\mathcal{O}}_{\mathcal{N}^{o},z}$ be the local equation of the difference divisor $\mathcal{D}(p^{a}x)$ at $z$. Notice that the point $z\in\mathcal{D}(p^{a}x)(\mathbb{F})$ if and only if $d_{p^{a}x}\in\mathfrak{m}_z$. Since $z\notin\mathcal{Z}(p^{-1}x)$, we can assume $f_x=d_x$.
\par
We first prove $d_x\in \mathfrak{m}_{z}\backslash\mathfrak{m}_{z}^{2}$, which is equivalent to $f_x\in \mathfrak{m}_{z}\backslash\mathfrak{m}_{z}^{2}$. There are two cases.
\begin{itemize}
    \item[$\bullet$] If $n_{z,x}\leq1$. Notice that $l_{z,x}=1$ since $z\notin\mathcal{Z}(p^{-1}x)(\mathbb{F})$. By Proposition \ref{comparison}, $l_{z,x}\leq t_{x}(\Tilde{z})\leq n_{z,x}$ for all points $\Tilde{z}\in\widehat{\mathcal{N}}^{o}_z(\zpb)$, then we must have $l_{z,x}= t_{x}(\Tilde{z})= n_{z,x}=1$. By Lemma \ref{orderlemma}, the number $\nu_{p}(\Tilde{z}^{\sharp}(f_{x}))=t_x(\Tilde{z})=1$ for all $\zpb$-algebra homomorphism $\Tilde{z}^{\sharp}:\zpb[[t_{1},\cdots,t_{n-1}]]\rightarrow \zpb$, hence $f_{x}\equiv (\textup{unit})\cdot p \,\,(\textup{mod}\,\,\mathfrak{m}_{z}^{2})$ by Lemma \ref{tech1}, then $d_x=f_x\in \mathfrak{m}_{z}\backslash\mathfrak{m}_{z}^{2}$. Notice that in this case the difference divisor $\mathcal{D}(x)$ is regular but not formally smooth over $\zpb$ at $z$.
    \item[$\bullet$] If $n_{z,x}\geq2$. Notice that we already have $l_{z,x}=1$, then by Lemma \ref{smooth-lemma-2} the formal scheme $\mathcal{Z}(x)$ is smooth over $\zpb$ of relative dimension $n-2$ at the point $z$, hence $f_{x}\in \mathfrak{m}_{z}\backslash\mathfrak{m}_{z}^{2}$. Notice that in this case the difference divisor $\mathcal{D}(x)$ is formally smooth over $\zpb$ at $z$.
\end{itemize}
\par
Now we already show that $d_x\in \mathfrak{m}_{z}\backslash\mathfrak{m}_{z}^{2}$. Our next goal is to prove $d_{p^{a}x}\equiv(\textup{unit})\cdot p\,\textup{mod}\,\mathfrak{m}_{z}^{2}$ for all $a\geq1$. Let $\Tilde{z}\in\widehat{\mathcal{N}}^{o}_z(\zpb)$ be a point, by Lemma \ref{orderlemma},
\begin{equation}
    (\Tilde{z}^{\sharp}(f_{p^{a}x}))=(\langle p^{a}\cdot x_{\textup{crys,z}},l\rangle)=(p^{a}\cdot\langle x_{\textup{crys,z}},l\rangle)=(p^{a}\cdot\Tilde{z}^{\sharp}(f_{x})).
    \label{key-cartesian-o}
\end{equation}
Let $\widehat{\mathcal{Z}(p^{a}x)}_{z}$ be the completion of the formal scheme $\mathcal{Z}(p^{a}x)$ at $z$, there exists a Cartesian diagram by (\ref{key-cartesian-o}),
\begin{equation*}
     \xymatrix{
    \textup{Spf}\,\zpb/(p^{a}\cdot\Tilde{z}^{\sharp}(f_{x}))\ar[r]\ar[d]&\textup{Spf}\,\zpb\ar[d]^{\tilde{z}}\\
    \widehat{\mathcal{Z}(p^{a}x)}_{z}\ar[r]&\widehat{\mathcal{N}}^{o}_{z}.}
\end{equation*}
Recall that by the definition of difference divisors, we have $f_{p^{a}x}=\prod\limits_{i=0}^{a}d_{p^{i}x}$ for all integer $a\geq0$, hence Lemma \ref{technical} implies that $d_{p^{a}x}\equiv(\textup{unit})\cdot p \,\,(\textup{mod}\,\,\mathfrak{m}_{z}^{2})$ for all $a\geq1$ since $d_{x}\in \mathfrak{m}_{z}\backslash\mathfrak{m}_{z}^{2}$.\par
For all integer $a\geq0$, the point $z\in\mathcal{D}(p^{a}x)(\mathbb{F})$ because the local equation $d_{p^{a}x}\in\mathfrak{m}_z$. Moreover, the local equation $d_{x}\in \mathfrak{m}_{z}\backslash\mathfrak{m}_{z}^{2}$ implies that the difference divisor $\mathcal{D}(p^{a}x)$ is regular at the point $z$.
\par
For $a=0$, we have proved that the difference divisor $\mathcal{D}(x)$ is formally smooth over $\zpb$ at $z$ if and only if $n_{z,x}\geq2$.
For all $a\geq1$, the equation $d_{p^{a}x}\equiv(\textup{unit})\cdot p \,\,(\textup{mod}\,\,\mathfrak{m}_{z}^{2})$, hence the difference divisor $\mathcal{D}(p^{a}x)$ is not formally smooth at the point $z$.
\par
Now we have shown (ii), (iii) and (iv), let's come back to prove (i). The one-side inclusion $\mathcal{D}(x)(\mathbb{F})\subset\mathcal{Z}(x)(\mathbb{F})$ follows from definition. Conversely, if $z^{\prime}\in\mathcal{Z}(x)(\mathbb{F})$. Let $x^{\prime}=p^{-l_{z^{\prime},x}+1}x$, then $z^{\prime}\in\mathcal{Z}(x^{\prime})(\mathbb{F})$ but $z^{\prime}\notin\mathcal{Z}(p^{-1}x^{\prime})(\mathbb{F})$. (ii) implies that $z^{\prime}\in\mathcal{D}(p^{l_{z^{\prime},x}-1}x^{\prime})(\mathbb{F})=\mathcal{D}(x)(\mathbb{F})$. Therefore we conclude that $\mathcal{D}(x)(\mathbb{F})=\mathcal{Z}(x)(\mathbb{F})$.\hfill$\qed$
\begin{example}
    We give some examples of the equality $\mathcal{D}(x)(\mathbb{F})=\mathcal{Z}(x)(\mathbb{F})$.
    \begin{itemize}
        \item [(a)] The orthogonal case: Let $H$ be the unique rank $4$ self-dual quadratic lattice over $\zp$ whose discriminant is not a square in $\zp^{\times}$. Let $x\in\mathbb{V}^{o}$ be a non-isotropic vector and $n=\nu_p(q_{\mathbb{V}^{o}}(x))$. Let $x^{\prime}=p^{-[n/2]}\cdot x$. Let $\mathscr{B}$ be the Bruhat-Tits building of $\textup{PGL}_2(\mathbb{Q}_{p^{2}})$. The element $x^{\prime}$ induces an automorphism of $\mathscr{B}$. Let $\mathscr{B}^{x^{\prime}}$ be the fixed set. For a vertex $[\Lambda]$ in the Bruhat-Tits building $\mathscr{B}$ of $\textup{PGL}_2(\mathbb{Q}_{p^{2}})$, let $d_{[\Lambda]}$ be the distance of $[\Lambda]$ to $\mathscr{B}^{x^{\prime}}$.
        \par
        Let $\mathcal{Z}(x)_{\mathbb{F}}=\mathcal{Z}(x)\times_{\zpb}\mathbb{F}$ and $\mathcal{D}(x)_{\mathbb{F}}=\mathcal{D}(x)\times_{\zpb}\mathbb{F}$. By \cite[Theorem 3.6]{Ter11}, we have the following equality of effective Cartier divisors on $\mathcal{N}^{o}_{\mathbb{F}}\coloneqq\mathcal{N}^{o}\times_{W}\mathbb{F}$,
        \begin{equation*}
        \mathcal{Z}(x)_{\mathbb{F}}=\sum\limits_{[\Lambda]:d_{[\Lambda]}\leq(n-1)/2}(1+p+\cdots+p^{(n-1)/2-d_{[\Lambda]}})\cdot\mathbb{P}_{[\Lambda]}+p^{n/2}\cdot s,
        \end{equation*}
        where $\mathbb{P}_{[\Lambda]}$ is a divisor associated to the vertex $[\Lambda]$ and isomorphic to the projective line $\mathbb{P}_{\mathbb{F}}^{1}$. The symbol $s$ represents another divisor which only depends on the element $x^{\prime}$ and is nonempty only if $n$ is even. Since $\mathcal{D}(x)=\mathcal{Z}(x)-\mathcal{Z}(p^{-1}x)$, we have
        \begin{equation*}
            \mathcal{D}(x)_{\mathbb{F}}=\begin{cases}
                \sum\limits_{[\Lambda]:d_{[\Lambda]}=(n-1)/2}\mathbb{P}_{[\Lambda]}+\sum\limits_{[\Lambda]:d_{[\Lambda]}\leq(n-3)/2}p^{(n-1)/2-d_{[\Lambda]}}\cdot\mathbb{P}_{[\Lambda]}+(p^{n/2}-p^{(n-2)/2})\cdot s, &\textup{if $n\geq1$;}\\
                s,&\textup{if $n=0$.}
            \end{cases}
        \end{equation*}
        Therefore $\mathcal{D}(x)(\mathbb{F})=\mathcal{Z}(x)(\mathbb{F})$.
        \item[(b)] The unitary case: Let $n=3$ and $F=\mathbb{Q}_p$. For an element $x\in\mathbb{V}^{u}$, the proof of (\cite[Proposition 2.13]{Ter13a}) implies that if $z\in\mathcal{Z}(x)(\mathbb{F})$, then the local equation of the difference divisor $\mathcal{D}(x)$ is not a unit at the point $z$, hence $\mathcal{Z}(x)(\mathbb{F})\subset\mathcal{D}(x)(\mathbb{F})$. Therefore $\mathcal{D}(x)(\mathbb{F})=\mathcal{Z}(x)(\mathbb{F})$.
    \end{itemize}
    \label{d=z}
\end{example}

\begin{corollary}
Let $x\in\mathbb{V}^{u}$ (resp. $x\in\mathbb{V}^{o}$) be an element such that $(x,x)\in\mathcal{O}_{E}$ (resp. $q_{\mathbb{V}^{o}}(x)\in\mathcal{O}_{F}$).
\begin{itemize}
    \item[(i)] The difference divisor $\mathcal{D}(x)$ is a regular formal scheme.
    \item[(ii)] The formal scheme $\mathcal{D}(x)$ is formally smooth over $\ofb$ (resp. $\zpb$) of relative dimension $n-2$ at a point $z\in\mathcal{D}(x)(\mathbb{F})$ if and only if $l_{z,x}=1$ and $n_{z,x}\geq2$.
    \item[(iii)] Let $z\in\mathcal{Z}(x)(\mathbb{F})$ be a point. Let $\mathcal{D}(x)_{\mathbb{F}}=\mathcal{D}(x)\times_{\ofb}\mathbb{F}$ and $\mathcal{Z}(x)_{\mathbb{F}}=\mathcal{Z}(x)\times_{\ofb}\mathbb{F}$ (resp. $\mathcal{D}(x)_{\mathbb{F}}=\mathcal{D}(x)\times_{\zpb}\mathbb{F}$ and $\mathcal{Z}(x)_{\mathbb{F}}=\mathcal{Z}(x)\times_{\zpb}\mathbb{F}$). Then 
    \begin{equation*}
        \textup{Tgt}_z(\mathcal{D}(x)_{\mathbb{F}})=\textup{Tgt}_z(\mathcal{Z}(x)_{\mathbb{F}}),
    \end{equation*}
    where $\textup{Tgt}_z(\mathcal{D}(x)_{\mathbb{F}})$ and $\textup{Tgt}_z(\mathcal{Z}(x)_{\mathbb{F}})$ are the tangent spaces of the formal scheme $\mathcal{D}(x)_{\mathbb{F}}$ and $\mathcal{Z}(x)_{\mathbb{F}}$ at the point $z$, respectively.
    \end{itemize}
\label{finalproof}
\end{corollary}
\begin{proof}
Let $x^{\prime}=\varpi^{-l_{z,x}+1}x$ throughout this proof. Notice that $z\in\mathcal{Z}(x^{\prime})(\mathbb{F})$ but $z\notin\mathcal{Z}(\varpi^{-1}x^{\prime})(\mathbb{F})$ (resp. $z\notin\mathcal{Z}(p^{-1}x^{\prime})(\mathbb{F})$).
\par
We first prove (i). Let $z\in\mathcal{D}(x)(\mathbb{F})$ be a point. By Theorem \ref{main-results} (iii) and (iv), the difference divisor $\mathcal{D}(x)=\mathcal{D}(\varpi^{l_{z,x}-1}x^{\prime})$ (resp. $\mathcal{D}(x)=\mathcal{D}(p^{l_{z,x}-1}x^{\prime})$) is regular at $z$. Since $z$ is an arbitrary point, we conclude that the formal scheme $\mathcal{D}(x)$ is regular.
\par
Now we prove (ii). Suppose that the difference divisor $\mathcal{D}(x)$ is formally smooth over $\ofb$ (resp. $\zpb$) at the point $z\in\mathcal{D}(x)(\mathbb{F})$. If $l_{z,x}\geq2$, the difference divisor $\mathcal{D}(x)=\mathcal{D}(\varpi^{l_{z,x}-1}x^{\prime})$ (resp. $\mathcal{D}(x)=\mathcal{D}(p^{l_{z,x}-1}x^{\prime})$) is not smooth over $\ofb$ (resp. $\zpb$) by Theorem \ref{main-results} (ii), hence we must have $l_{z,x}=1$. In the case $l_{z,x}=1$, we get $n_{z,x}\geq2$ by Theorem \ref{main-results} (iii). For the converse direction. Let's assume $l_{z,x}=1$ and $n_{z,x}\geq2$. Theorem \ref{main-results} (iii) implies that the difference divisor $\mathcal{D}(x)$ is formally smooth over $\ofb$ (resp. $\zpb$) at $z$. This concludes the whole proof.
\par
Finally we prove (iii). Let $\widehat{\mathcal{D}(x)}_z$ and $\widehat{\mathcal{Z}(x)}_z$ be the completion of the formal scheme $\mathcal{D}(x)$ and $\mathcal{Z}(x)$ at the point $z$ respectively. By Schlessinger's criterion in \cite[Theorem 2.11]{Sch68},
\begin{equation*}
    \textup{Tgt}_z(\mathcal{D}(x)_{\mathbb{F}})=\widehat{\mathcal{D}(x)}_z(\mathbb{F}[\epsilon]),\,\,\textup{Tgt}_z(\mathcal{Z}(x)_{\mathbb{F}})=\widehat{\mathcal{Z}(x)}_z(\mathbb{F}[\epsilon])
\end{equation*}
where $\mathbb{F}[\epsilon]=\mathbb{F}[X]/(X^2)$. The closed immersion $\mathcal{D}(x)\rightarrow\mathcal{Z}(x)$ implies that there is an injection $\textup{Tgt}_z(\mathcal{D}(x)_{\mathbb{F}})\rightarrow\textup{Tgt}_z(\mathcal{Z}(x)_{\mathbb{F}})$. We want to show that this is bijective. We separate the proof into two cases.
\begin{itemize}
    \item[(a)] The point $z\notin\mathcal{Z}(\varpi^{-1}x)(\mathbb{F})$ (resp. $z\notin\mathcal{Z}(p^{-1}x)(\mathbb{F})$). Then we have $\widehat{\mathcal{D}(x)}_z\simeq\widehat{\mathcal{Z}(x)}_z$. The bijitivity is obvious.
    \item[(b)] The point $z\in\mathcal{Z}(\varpi^{-1}x)(\mathbb{F})$ (resp. $z\in\mathcal{Z}(p^{-1}x)(\mathbb{F})$). Let $d_{x,z}\in\widehat{\mathcal{O}}_{\mathcal{N}^{u},z}$ (resp. $d_{x,z}\in\widehat{\mathcal{O}}_{\mathcal{N}^{o},z}$) be the local equation of the difference divisor. Since $l_{z,x}\geq2$, the difference divisor $\mathcal{D}(x)$ is not formally smooth over $\ofb$ (resp. $\zpb$) at the $z$, hence $d_{x,z}\equiv\varpi$\,mod $\mathfrak{m}_z^{2}$ (resp. $d_{x,z}\equiv p$\,mod $\mathfrak{m}_z^{2}$). Let $\Tilde{z}\in\widehat{\mathcal{N}}^{u}_z(\mathbb{F}[\epsilon])$ (resp. $\Tilde{z}\in\widehat{\mathcal{N}}^{o}_z(\mathbb{F}[\epsilon])$) be a point. Let $\Tilde{z}^{\sharp}:\widehat{\mathcal{O}}_{\mathcal{N}^{u},z}\rightarrow\mathbb{F}[\epsilon]$ (resp. $\Tilde{z}^{\sharp}:\widehat{\mathcal{O}}_{\mathcal{N}^{o},z}\rightarrow\mathbb{F}[\epsilon]$) be the corresponding ring homomorphism. We have $\Tilde{z}^{\sharp}(d_{x,z})=0$ because $\Tilde{z}^{\sharp}(\varpi)=0$ (resp. $\Tilde{z}^{\sharp}(p)=0$) and $\Tilde{z}^{\sharp}(\mathfrak{m}_z^{2})\subset(\epsilon^{2})=0$. Hence $\widehat{\mathcal{D}(x)}_z(\mathbb{F}[\epsilon])=\widehat{\mathcal{N}}^{u}_z(\mathbb{F}[\epsilon])\supset\widehat{\mathcal{Z}(x)}_z(\mathbb{F}[\epsilon])$ (resp. $\widehat{\mathcal{D}(x)}_z(\mathbb{F}[\epsilon])=\widehat{\mathcal{N}}^{o}_z(\mathbb{F}[\epsilon])\supset\widehat{\mathcal{Z}(x)}_z(\mathbb{F}[\epsilon])$). Therefore $\widehat{\mathcal{D}(x)}_z(\mathbb{F}[\epsilon])=\widehat{\mathcal{Z}(x)}_z(\mathbb{F}[\epsilon])$.
\end{itemize}
Therefore we have proved that $\textup{Tgt}_z(\mathcal{D}(x)_{\mathbb{F}})=\textup{Tgt}_z(\mathcal{Z}(x)_{\mathbb{F}})$.
\end{proof}
~\\~\\
\noindent\textbf{Conflict of interest statement.}\,\,Baiqing Zhu declares no conflicts of interest.\\
\noindent\textbf{Data availability statement.}\,\,No new data were created or analysed in this study. Data sharing is not applicable to this article.

\bibliographystyle{alpha}
\bibliography{reference}

\end{document}